\documentclass[11pt,pdftex, reqno]{amsart}
\pdfoutput=1

\usepackage{amsmath, amsthm}
\usepackage{eucal}
\usepackage{palatino}
\usepackage{euler}
\usepackage{amssymb}
\usepackage{amscd}
\usepackage{latexsym}
\usepackage{epsfig}
\usepackage{graphicx, xcolor}
\usepackage{amsfonts}
\usepackage{psfrag}
\usepackage{caption}
\usepackage{setspace}
\usepackage{mathabx}
\usepackage{fullpage}
\usepackage{draftcopy}
\usepackage{tikz}
\usepackage[normalem]{ulem}
\usepackage{pifont}
\usepackage{verbatim}
\usepackage{thmtools}
\usepackage{thm-restate} 
\usepackage{tabmac}
\usepackage{hyperref}
\usepackage{dsfont}
\usepackage{url}
\usepackage{cite}

\usepackage[margin=1in, headsep=.25in]{geometry}

\input xy
\xyoption{all}
\UseComputerModernTips

\numberwithin{equation}{section}
\numberwithin{figure}{section}

\newtheorem{theorem}{Theorem}[section]

\newtheorem{lemma}[theorem]{Lemma}
\newtheorem{proposition}[theorem]{Proposition}

\theoremstyle{definition}
\newtheorem{definition}[theorem]{Definition}
\newtheorem{remark}[theorem]{Remark}
\newtheorem{example}[theorem]{Example}

\newcommand{\C}{{\mathbb{C}}}

\newcommand{\Z}{{\mathbb{Z}}}

\newcommand{\N}{{\mathbb{N}}}

\newcommand{\ba}{\backslash}

\newcommand{\wtv}{\operatorname{wt}^v}
\newcommand{\wth}{\operatorname{wt}^h}

\newcommand\sbullet[1][.5]{\mathbin{\vcenter{\hbox{\scalebox{#1}{$\bullet$}}}}}
\newcommand\sasterisk[1][.5]{\mathbin{\vcenter{\hbox{\scalebox{#1}{$\coAsterisk$}}}}}

\newcommand\sstar[1][.5]{\mathbin{\vcenter{\hbox{\scalebox{#1}{$\star$}}}}}

\definecolor{limegreen}{rgb}{0.2, 0.8, 0.2}
\definecolor{amethyst}{rgb}{0.6, 0.4, 0.8}
\definecolor{deeppink}{rgb}{1.0, 0.08, 0.58}

\newcommand{\red}{\textcolor{red}}

\newcommand{\green}{\textcolor{limegreen}}
\newcommand{\blue}{\textcolor{blue}}

\begin{document}

\pagestyle{headings}

\title[An Equivariant Quantum Pieri Rule on Cylindric Shapes]{An Equivariant Quantum Pieri Rule for the Grassmannian on Cylindric Shapes}

\author{Anna Bertiger}
\address{Anna Bertiger, Microsoft, Seattle, WA}
\email{anberti@microsoft.com}
\author{Dorian Ehrlich}
\address{Dorian Ehrlich, New York, NY}
\email{dehrlich014@gmail.com}
\author{Elizabeth Mili\'cevi\'c}
\address{Elizabeth Mili\'cevi\'c, Haverford College, Haverford, PA}
\email{emilicevic@haverford.edu}
\author{Kaisa Taipale}
\address{Kaisa Taipale, C.~H.~Robinson, Eden Prairie, MN}
\email{taipale@umn.edu}

\thanks{AB was supported by the University of Waterloo while working on this project. DE was fully supported by NSF Grant DMS-1600982. EM was partially supported by NSF Grant DMS-1600982, Simons Collaboration Grant 318716, and the Max-Planck-Institut f\"ur Mathematik. KT was supported by University of Minnesota while working on this project.}

\begin{abstract}
The quantum cohomology ring of the Grassmannian is determined by the quantum Pieri rule for multiplying by Schubert classes indexed by row or column-shaped partitions.  We provide a direct equivariant generalization of Postnikov's quantum Pieri rule for the Grassmannian in terms of cylindric shapes, complementing related work of Gorbounov and Korff in quantum integrable systems. The equivariant terms in our Graham-positive rule simply encode the positions of all possible addable boxes within one cylindric skew diagram.  As such, unlike the earlier equivariant quantum Pieri rule of Huang and Li and known equivariant quantum Littlewood-Richardson rules, our formula does not require any calculations in a different Grassmannian or two-step flag variety.
\end{abstract}

\maketitle

%%%%%%%%%%%%%%%%%%%%%%%%%%%%%%%%%%%%%%%%%%%%
%%%%%%%%%%%%%%%%%%%%%%%%%%%%%%%%%%%%%%%%%%%%
%%%%%%%%%%%%%%%%%%%%%%%%%%%%%%%%%%%%%%%%%%%%
%%%%%%%%%%%%%%%%%%%%%%%%%%%%%%%%%%%%%%%%%%%%
%%%%%%%%%%%%%%%%%%%%%%%%%%%%%%%%%%%%%%%%%%%%

\section{Introduction}\label{sec:intro}

Equivariant cohomology is a powerful tool for studying the geometry and topology of algebraic varieties which admit an action by a reductive algebraic group; see the survey \cite{Anderson}. With key ingredients appearing in Cartan's complex of equivariant differential forms \cite{Cartan}, the Borel construction for equivariant singular cohomology dates back to the 1950's \cite{Borel}. The prevailing philosophy in the equivariant setting is to exploit the symmetries arising from the group action by replacing topological questions with a limited amount of algebraic information.  For example, the theory of localization following \cite{GKM} permits the equivariant cohomology of nice topological spaces, such as flag varieties and Grassmannians, to be represented using a graph determined by the fixed points and one-dimensional orbits of the group action; see \cite{GZ} for an overview. Kim's construction of the equivariant Gromov-Witten invariants \cite{Kim} extends these principles to the (small) equivariant quantum cohomology of flag varieties, and thus to modern equivariant quantum Schubert calculus via various equivariant quantum-to-classical phenomena \cite{BuchMihalcea,BMT}.

\subsection{Equivariant quantum Schubert calculus}

The equivariant quantum cohomology ring for the Grassmannian $Gr(m,n)$ has a basis of \emph{Schubert classes}, indexed by partition shapes which fit inside an $m \times (n-m)$ rectangle.  The fundamental question in \emph{equivariant quantum Schubert calculus} is to find combinatorial formulas for the expansion of the product of two Schubert classes in terms of this Schubert basis.  A Graham-positive \emph{equivariant quantum Littlewood-Richardson rule} for $QH^*_T(Gr(m,n))$ is now available, by combining Buch's generalization of Knutson and Tao's equivariant puzzle rule \cite{KT} to two-step flag varieties \cite{Buch2step}, together with Buch and Mihalcea's quantum-to-classical principle equating the equivariant quantum Littlewood-Richardson coefficients for $Gr(m,n)$ to their classical counterparts in a related two-step flag variety \cite{BuchMihalcea}.  

Other formulas for the equivariant quantum Littlewood-Richardson coefficients exist, the first of which was Mihalcea's \emph{equivariant quantum Pieri-Chevalley rule} for multiplying by the Schubert class corresponding to a single box \cite{MihalceaTrans}, which recursively determines the ring structure of $QH^*_{T}(Gr(m,n))$.  Any of the numerous equivariant Littlewood-Richardson rules for $H^*_T(Gr(m,n))$ can be combined with the \emph{equivariant rim hook rule} of the first, third, and fourth authors \cite{BMT} to produce a signed method for computing products in $QH^*_{T}(Gr(m,n))$. Gorbounov and Korff take an integrable systems approach to provide an explicit determinantal formula for the equivariant quantum Littlewood-Richardson coefficients \cite{GKunpub,GK}.

Without performing calculations in a related two-step flag variety, the most general Graham-positive combinatorial formula for the equivariant quantum Littlewood-Richardson coefficients is the \emph{equivariant quantum Pieri rule} for multiplying by a Schubert class indexed by a row or column-shaped partition. Huang and Li proved the first equivariant quantum Pieri rule for any type $A$ partial flag variety in Theorem 3.10 of \cite{HuangLi}, though in the the special case of the Grassmannian, this formula can also be recovered from Robinson's earlier full flag Pieri rule \cite{Robinson}, for which Li, Ravikumar, Sottile, and Yang later provided an alternate geometric proof \cite{LRSY}. We remark that Li and Ravikumar have also generalized the underlying equivariant Pieri rule for the Grassmannian to other classcial Lie types \cite{LiRav}. The critical component in each of the aforementioned Pieri rules is the interpretation of certain equivariant quantum Littlewood-Richardson coefficients as \emph{localizations} at a given $T$-fixed point of a Schubert variety in a  smaller Grassmannian.  Gorbounov and Korff take a different approach in \cite{GKunpub,GK}, providing an equivariant quantum Pieri rule in terms of an action by vicious and osculating walker transfer matrices in the Yang-Baxter algebra.

\begin{figure}[h]

\[
\begin{tikzpicture}[scale=.45]

\draw[-]

(10,10) -- (12,10) -- (12,7) -- (10,7) -- cycle
;

\draw[color=cyan]
(14,13) -- (14,12) -- (13,12) -- (13,11) -- (12,11) -- (12,10) -- (12,9) -- (11,9) -- (11,8) -- (10,8) -- (10,7) -- (10,6) -- (9,6) -- (9,5) -- (8,5) -- (8,4)
;

\draw[color=red]
(14.1,12.9) -- (14.1,11.9) -- (13.1,11.9) -- (13.1,10.9) -- (13.1,9.9) -- (12.1,9.9) -- (12.1,8.9) -- (11.1,8.9) -- (11.1,7.9) -- (11.1,6.9) -- (10.1,6.9) -- (10.1,5.9) -- (9.1,5.9) -- (9.1,4.9) -- (9.1,3.9) -- (8.1,3.9)
;

\draw[color=green]
(10,9) -- (11,9) -- (11,10);

\node[color = limegreen] at (10.5,8.5) {$\sbullet[1.5]$};
\node[color = limegreen] at (11.5,9.5) {$\sbullet[1.5]$};

\draw[dashed,color = gray]
(8,10) -- (10,7) -- (12,4)
(10,13) -- (12,10) -- (14,7);

\node[scale = .6, color = cyan] at (7.4,4.5) {$\mu[0]$};
\node[scale = .6, color = red] at (9.7,4.5) {$\lambda[0]$};

\path[draw,color = white, fill = red!40] (8.1,4) -- (8.1,4.9) -- (9,4.9) -- (9,4) -- cycle;
\path[draw,color = white, fill = red!40] (10.1,7.1) -- (10.1,7.9) -- (11,7.9) -- (11,7.1) -- cycle;
\path[draw,color = white, fill = red!40] (12.1,10) -- (12.1,10.9) -- (13,10.9) -- (13,10) -- cycle;

\node[scale = .6,color = blue] at (9.8,7.5) {$1$};
\node[scale = .6,color = blue] at (10.5,8.05) {$2$};
\node[scale = .6,color = blue] at (11.35,8.35) {$3$};
\node[scale = .6,color = blue] at (11.75,8.65) {$4$};
\node[scale = .6,color = blue] at (12.3,9.35) {$5$};

\node at (7,9) {$\downarrow$};
\node at (9,12) {$\rightarrow$};
\node[scale = .6] at (9.5,11.8) {$j$};
\node[scale = .6] at (7.3,8.8) {$i$};

\node[scale = .6] at (9.25,10) {$(0,0)$};
\node[scale = .6] at (12.7,7) {$(3,2)$};

\node[color = blue] at (8,4) {$\bullet$};
\node[color =  blue] at (14,13) {$\bullet$};

\node[color = blue] at (12,10) {$\bullet$};
\node[color =  blue] at (10,7) {$\bullet$};

\node at (10,10) {$\bullet$};
\node at (12,7) {$\bullet$};

%%%%%%%%%%%%%%%%%%%%%%%%%%%%%%%%%%%%%%%%%%%%

\draw[-]

(17,10) -- (19,10) -- (19,7) -- (17,7) -- cycle
;

\draw[color=cyan]
(21,13) -- (21,12) -- (20,12) -- (20,11) -- (19,11) -- (19,10) -- (19,9) -- (18,9) -- (18,8) -- (17,8) -- (17,7) -- (17,6) -- (16,6) -- (16,5) -- (15,5) -- (15,4)
;

\draw[color=red]
(21.1,13.1) -- (21.1,12.1) -- (21.1,11) -- (20,11) -- (20,10.1) -- (19.1,10.1) -- (19.1,9.1) -- (19.1,8) -- (18,8) -- (18,7.1) -- (17.1,7.1) -- (17.1,6.1) -- (17.1,5) -- (16,5) -- (16,4.1) -- (15.1,4.1)
;

\draw[color = green]
(18,10) -- (18,9);

\draw[dashed,color = gray]
(15,10) -- (17,7) -- (19,4)
(17,13) -- (19,10) -- (21,7)
;

\node[scale = .6] at (16.2,10) {$(0,0)$};
\node[scale = .6] at (19.8,7) {$(3,2)$};

\node[color = blue] at (15,4) {$\bullet$};
\node[color =  blue] at (21,13) {$\bullet$};

\node[color = blue] at (19,10) {$\bullet$};
\node[color =  blue] at (17,7) {$\bullet$};

\node at (17,10) {$\bullet$};
\node at (19,7) {$\bullet$};

\node[color = limegreen] at (18.5,9.5) {$\sbullet[1.5]$};

\path[draw,color = white, fill = red!40] (15.1,4.2) -- (15.1,4.9) -- (15.9,4.9) -- (15.9,4.2) -- cycle;
\path[draw,color = white, fill = red!40] (16.1,5.1) -- (16.1,5.9) -- (17,5.9) -- (17,5.1) -- cycle;
\path[draw,color = white, fill = red!40] (17.1,7.2) -- (17.1,7.9) -- (17.9,7.9) -- (17.9,7.2) -- cycle;
\path[draw,color = white, fill = red!40] (17.1,7.2) -- (17.1,7.9) -- (17.9,7.9) -- (17.9,7.2) -- cycle;
\path[draw,color = white, fill = red!40] (18.1,8.1) -- (18.1,8.9) -- (18.9,8.9) -- (18.9,8.1) -- cycle;
\path[draw,color = white, fill = red!40] (19.1,10.15) -- (19.1,10.9) -- (19.9,10.9) -- (19.9,10.15) -- cycle;
\path[draw,color = white, fill = red!40] (20.1,11.1) -- (20.1,11.9) -- (21,11.9) -- (21,11.1) -- cycle;

\node[scale = .6, color = blue] at (16.8,7.5) {$1$};
\node[scale = .6, color = blue] at (17.5,8.05) {$2$};
\node[scale = .6, color = blue] at (17.8,8.5) {$3$};
\node[scale = .6, color = blue] at (18.5,9) {$4$};
\node[scale = .6, color = blue] at (19.3,9.35) {$5$};

\node[scale = .6,color = cyan] at (14.4,4.5) {$\mu[0]$};
\node[scale = .6,color = red] at (16.6,4.5) {$\lambda[0]$};

%%%%%%%%%%%%%%%%%%%%%%%%%%%%%%%%%%%%%%%%%%%%

\draw[-]

(24,10) -- (26,10) -- (26,7) -- (24,7) -- cycle
;

\draw[color=cyan]
(28,13) -- (28,12) -- (27,12) -- (27,11) -- (26,11) -- (26,10) -- (26,9) -- (25,9) -- (25,8) -- (24,8) -- (24,7) -- (24,6) -- (23,6) -- (23,5) -- (22,5) -- (22,4)
;

\draw[color=red]
(29.1,13) -- (29.1,11.9) -- (28.1,11.9) -- (27.1,11.9) -- (27.1,10.9) -- (27.1,9.9) -- (27.1,8.9) -- (26.1,8.9) -- (25.1,8.9) -- (25.1,7.9) -- (25.1,6.9) -- (25.1,5.9) -- (24.1,5.9) -- (23.1,5.9) -- (23.1,4.9) -- (23.1,3.9)
;

\draw[color = green]
(25,9) -- (24,9)
;

\path[draw,color = white, fill = red!40] (22.1,4.1) -- (22.1,4.9) -- (23,4.9) -- (23,4.1) -- cycle;
\path[draw,color = white, fill = red!40] (24.1,6) -- (24.1,6.9) -- (25,6.9) -- (25,6) -- cycle;
\path[draw,color = white, fill = red!40] (24.1,7.1) -- (24.1,7.9) -- (25,7.9) -- (25,7.1) -- cycle;
\path[draw,color = white, fill = red!40] (26.1,9) -- (26.1,9.95) -- (27,9.95) -- (27,9) -- cycle;
\path[draw,color = white, fill = red!40] (26.1,10) -- (26.1,10.9) -- (27,10.9) -- (27,10) -- cycle;
\path[draw,color = white, fill = red!40] (28.1,12) -- (28.1,12.9) -- (29,12.9) -- (29,12) -- cycle;

\node[scale = .6] at (23.2,10) {$(0,0)$};
\node[scale = .6] at (26.8,7) {$(3,2)$};

\node[color = blue] at (22,4) {$\bullet$};
\node[color =  blue] at (28,13) {$\bullet$};

\node[color = blue] at (26,10) {$\bullet$};
\node[color =  blue] at (24,7) {$\bullet$};

\node at (24,10) {$\bullet$};
\node at (26,7) {$\bullet$};

\node[color = limegreen] at (24.5,8.5) {$\sbullet[1.5]$};

\node[color = cyan,scale = .6] at (21.4,4.5) {$\mu[0]$};
\node[color = red,scale = .6] at (23.7,4.5) {$\lambda[1]$};

\node[color = blue,scale = .6] at (23.8,7.5) {$1$};
\node[color = blue,scale = .6] at (24.5,8.05) {$2$};
\node[color = blue,scale = .6] at (25.3,8.4) {$3$};
\node[color = blue,scale = .6] at (25.55,9) {$4$};
\node[color = blue,scale = .6] at (25.8,9.5) {$5$};

\draw[dashed,color = gray]
(22,10) -- (24,7) -- (26,4)
(24,13) -- (26,10) -- (28,7);
 
%%%%%%%%%%%%%%%%%%%%%%%%%%%%%%%%%%%%%%%%%%%%

\draw[-]
(31,10) -- (33,10) -- (33,7) -- (31,7) -- cycle
;

\draw[color=cyan]
(35,13) -- (35,12) -- (34,12) -- (34,11) -- (33,11) -- (33,10) -- (33,9) -- (32,9) -- (32,8) -- (31,8) -- (31,7) -- (31,6) -- (30,6) -- (30,5) -- (29,5) -- (29,4)
;

\draw[color=red]
(36,13) -- (36,12) -- (35,12) -- (35,11) -- (34,11) -- (34,10) -- (34,9) -- (33,9) -- (33,8) -- (32,8) -- (32,7) -- (32,6) -- (31,6) -- (31,5) -- (30,5) -- (30,4)
;

\node[scale = .6] at (30.2,10) {$(0,0)$};
\node[scale = .6] at (33.8,7) {$(3,2)$};

\node at (31,10) {$\bullet$};
\node at (33,7) {$\bullet$};

\node[scale = .6,color = cyan] at (28.4,4.5) {$\mu[0]$};
\node[scale = .6,color = red] at (30.6,4.4) {$\lambda[1]$};

\path[draw,color = white, fill = red!40] (29.1,4.1) -- (29.1,4.9) -- (29.9,4.9) -- (29.9,4.1) -- cycle;
\path[draw,color = white, fill = red!40] (30.1,5.1) -- (30.1,5.9) -- (30.9,5.9) -- (30.9,5.1) -- cycle;
\path[draw,color = white, fill = red!40] (31.1,6.1) -- (31.1,6.9) -- (31.9,6.9) -- (31.9,6.1) -- cycle;
\path[draw,color = white, fill = red!40] (31.1,7.1) -- (31.1,7.9) -- (31.9,7.9) -- (31.9,7.1) -- cycle;
\path[draw,color = white, fill = red!40] (32.1,8.1) -- (32.1,8.9) -- (32.9,8.9) -- (32.9,8.1) -- cycle;
\path[draw,color = white, fill = red!40] (33.1,9.1) -- (33.1,9.95) -- (33.9,9.95) -- (33.9,9.1) -- cycle;
\path[draw,color = white, fill = red!40] (33.1,10.05) -- (33.1,10.9) -- (33.9,10.9) -- (33.9,10.05) -- cycle;
\path[draw,color = white, fill = red!40] (34.1,11.1) -- (34.1,11.9) -- (34.9,11.9) -- (34.9,11.1) -- cycle;
\path[draw,color = white, fill = red!40] (35.1,12.1) -- (35.1,12.9) -- (35.9,12.9) -- (35.9,12.1) -- cycle;

\node[color = blue] at (29,4) {$\bullet$};
\node[color =  blue] at (35,13) {$\bullet$};

\node[color = blue] at (33,10) {$\bullet$};
\node[color =  blue] at (31,7) {$\bullet$};

\node[scale = .6,color = blue] at (30.8,7.5) {$1$};
\node[scale = .6,color = blue] at (31.5,8.05) {$2$};
\node[scale = .6,color = blue] at (31.8,8.5) {$3$};
\node[scale = .6,color = blue] at (32.5,9) {$4$};
\node[scale = .6,color = blue] at (32.8,9.5) {$5$};

\draw[dashed,color = gray]
(29,10) -- (31,7) -- (33,4)
(31,13) -- (33,10) -- (35,7);

%%%%%%%%%%%%%%%%%%%%%%%%%%%%%%%%%%%%%%%%%%%%

;

\end{tikzpicture}
\]

\caption{Equivariant cylindric shapes to calculate $\sigma_{(1)^3} \star \sigma_{(2,1)} \in QH^*_T(Gr(3,5))$.}\label{fig:ABsintro}
\end{figure}
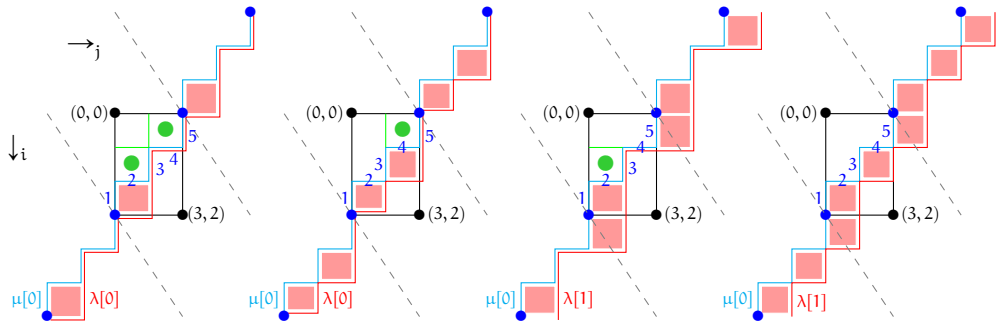

The goal of this paper is to provide a closed combinatorial Graham-positive equivariant quantum Pieri rule for $Gr(m,n)$, formulated in terms of an equivariant generalization of \emph{cylindric shapes}; see Figure \ref{fig:ABsintro} for an illustration which is fully explained in Example \ref{ex:thmexfull}.  Introduced by Gessel and Krattenthaler \cite{GesselKratt}, cylindric shapes have since appeared in many mathematical contexts.  In representation theory, cylindric diagrams enter Suzuki and Vazirani's work on representations of the double affine Hecke algebra \cite{SuzVaz}. In mathematical physics, Korff defines cylindric versions of specializations of skew Macdonald polynomials \cite{Korff}, which arise as partition functions of vertex models obtained by solving the Yang-Baxter equation. Arising from conjectures of McNamara \cite{McN} recently settled by Lee \cite{Lee}, the cylindric (skew) Schur functions are a family of symmetric functions indexed by cylindric shapes, and they play an important role in algebraic combinatorics. Central in Lee's work is the result of Lam that all cylindric Schur functions are affine Schur functions \cite{LamAJM}, which provides a connection to positroid varieties as studied by Knutson, Lam, and Speyer \cite{KLS}. Cylindric shapes also arise in previous works on the quantum cohomology of the Grassmannian, implicitly in the development of quantum Kostka numbers by Bertram, Ciocan-Fontanine, and Fulton \cite{BCFF}, as well as explicitly in the quantum Pieri rule of Postnikov \cite{Postnikov}; see Theorem \ref{thm:post} below. Gorbounov and Korff's version of the equivariant quantum Pieri rule translates the graphical calculus of non-intersecting lattice paths into the language of toric skew shapes  \cite{GKunpub,GK}.

\subsection{Statement of the main theorem}

Postnikov's version of the quantum Pieri rule provides an affine approach to $QH^*(Gr(m,n))$, and several additional quantum-to-affine connections from \cite{Postnikov} are solidified in Knutson, Lam, and Speyer's related work on positroid varieties \cite{KLS}. A key motivation for our search for particular equivariant generalizations of these quantum-to-affine phenomena is the parabolic Peterson isomorphism \cite{Pet,LamShimPet}, which implies that the equivariant homology of the affine Grassmannian surjects onto $QH^*_{T}(Gr(m,n))$, up to localization of the quantum parameter, in a manner which preserves the relevant Schubert calculus.  

Inspired by these affine connections, we present an equivariant generalization of the quantum Pieri rule from \cite{Postnikov}, or equivalently closed combinatorial formulas for the equivariant quantum Pieri expansions provided in \cite{GKunpub,GK}, as our main result.  In this context, the equivariant terms are generated by \emph{addable boxes} on the cylindric skew shape, which are depicted by the boxes with $\green{\sbullet[1.5]}$ in Figure \ref{fig:ABsintro} and whose contributions are denoted by $\alpha$ in Theorem \ref{thm:main} below. For the reader who is familiar with similar combinatorial formulas, the following is our main theorem statement.  All terminology is made precise in Section \ref{sec:background}, but here we describe the key components in the product as they relate to Figure \ref{fig:ABsintro}, with forward references to the relevant definitions in Section \ref{sec:background}.

\begin{restatable}[Equivariant Quantum Pieri Rule on Cylindric Shapes]{thm}{EqQPieri} 
\label{thm:main}
For any integers $1 \leq p \leq m$ and $1 \leq k \leq n-m$ and any partition $\mu \in P_{mn}$, we have
\begin{align*}
\sigma_{(1)^p} \star \sigma_\mu&=\sum\limits_{0 \leq r \leq p} \ \sum_{\substack{ \lambda/d/\mu = v^r  \\ v^r \rightarrow v^p  }} q^d\prod_{\alpha \in v^p  \ba v^r}\wtv(\alpha) \sigma_\lambda \quad \text{and} \\
\sigma_{(k)} \star \sigma_\mu&=\sum\limits_{0 \leq r \leq k} \ \sum_{\substack{ \lambda/d/\mu = h^r  \\ h^r \rightarrow h^k  }} q^d\prod_{\alpha \in h^k  \ba h^r}\wth(\alpha) \sigma_\lambda 
\end{align*}
in $QH^*_{T}(Gr(m,n)).$  Since $\lambda/d/\mu$ is a vertical (resp.~horizontal) strip, then $d \in \{0,1\}$. 
\end{restatable}

The inner sum in Theorem \ref{thm:main} ranges over cylindric skew diagrams $\lambda/d/\mu$ which are vertical (resp.~horizontal) strips, as depicted by the red shaded boxes in Figure \ref{fig:ABsintro}.  These vertical (resp.~horizontal) $r$-strips are then extended to full $p$- (resp.~$k$)-strips by adding boxes which preserve the nature of the original vertical (resp.~horizontal) strip, as do the boxes containing $\green{\sbullet[1.5]}$ in Figure \ref{fig:ABsintro}; see Definition \ref{def:AB} for details. The weight of each addable box is then easily calculated from several statistics which record the location of the box $\green{\sbullet[1.5]}$ within the portion of the skew shape contained in the ambient $m\times (n-m)$ rectangle, outlined in black in Figure \ref{fig:ABsintro}. We make the weight of an addable box precise in Definition \ref{def:wtaddbox}, and we explicitly record the weights corresponding to each of the four cylindric shapes appearing in Figure \ref{fig:ABsintro} in Example \ref{ex:thmexfull}. It is immediate from Definition \ref{def:wtaddbox} that both formulas in Theorem \ref{thm:main} are manifestly Graham-positive.

In comparison to other known equivariant quantum Pieri and/or Littlewood-Richardson rules, Theorem \ref{thm:main} more directly captures the quantum-to-affine phenomenon which governs the ring structure of $QH^*_{T}(Gr(m,n))$ via the parabolic Peterson isomorphism, as explicated in \cite{CM} in the non-equivariant case. In addition, each of the existing closed combinatorial formulas for these products typically requires doing calculations in many different related two-step flags or smaller Grassmannians in order to calculate a single Pieri product; compare \cite{HuangLi, BuchMihalcea, Buch2step}. We emphasize, by contrast, that Theorem \ref{thm:main} permits an entire equivariant quantum Pieri product to be carried out directly on the skew shapes $\lambda/d/\mu$ as in Figure \ref{fig:ABsintro}, which fully illustrates the product $\sigma_{\tableau[pcY]{\\ \\ \\ }} \star \sigma_{\tableau[pcY]{& \\ \\ }} \in QH^*_T(Gr(3,5))$; see Example \ref{ex:thmexfull}.  As such, even the classical case of Theorem \ref{thm:main} is quite distinct from existing formulations of the equivariant Pieri and Littlewood-Richardson rules for $H^*_{T}(Gr(m,n))$.  Moreover, as we explain in the next subsection, Theorem \ref{thm:main} provides an ideal starting point for establishing equivariant generalizations of many related phenomena in combinatorics, representation theory, and mathematical physics, rather than serving exclusively as an illuminating and convenient means to an equivariant quantum Schubert calculus end.

The proof of Theorem \ref{thm:main} proceeds by comparing the classical terms indexed by cylindric shapes with the equivariant Pieri rule from \cite{HuangLi} reviewed here as Theorem \ref{T:HuangLi}. More specifically, we prove in Proposition \ref{prop:PsiAB} that the weight determined by a configuration of addable boxes is the same whether calculated on the cylindric skew shape $\lambda/\mu$ or on the partition $\lambda_\mu$ defined in \cite{HuangLi} by a join-and-cut algorithm which passes to a smaller Grassmannian.  In Lemma \ref{lem:Psielem}, we then equate the equivariant Littlewood-Richardson coefficients appearing in Theorem \ref{thm:main} with certain specializations of elementary factorial Schur polynomials. These specialized factorial Schur polynomials agree with the classical terms appearing in Huang and Li's equivariant Pieri formula via the localization formula due to \cite{KT,LRS,IkedaNaruse} reviewed as Theorem \ref{thm:facSchur}.  Our equivariant quantum Pieri rule on cylindric shapes then follows by applying the equivariant rim hook rule of the first, third, and fourth authors\cite{BMT} reviewed as Theorem \ref{thm:eqrimhook}.

\subsection{Discussion of related work}
We now highlight some applications of Theorem \ref{thm:main} suggested by related literature in the non-equivariant setting. Although the appropriate equivariant analog is often a direct consequence of the combinatorics developed here, an in-depth pursuit of any of the generalizations discussed in this subsection is beyond the scope of the present paper.

Both \cite{Postnikov} and \cite{BCFF} define quantum Kostka numbers by iterating the quantum Pieri rule.  These quantum Kostka numbers are the coefficients in the monomial expansion of the cylindric Schur functions, a specialization of which give the toric Schur polynomials whose Schur expansions encode all quantum Littlewood-Richardson coefficients; see \cite[Theorem 5.3]{Postnikov}. Iterating the equivariant quantum Pieri rule in Theorem \ref{thm:main} thus yields an immediate equivariant generalization for each of these phenomena, which all carry important representation-theoretic meaning in the non-equivariant setting. For example, quantum Kostka numbers coincide with Wess-Zumino-Novikov-Witten (WZNW) fusion coefficients; compare the generalization \cite[Conjecture 1.1]{Korff} concerning cylindric skew Macdonald functions, special cases of which connect to Kostka-Foulkes polynomials, graded characters of Demazure modules of the current algebra, and Feigin-Loktev fusion products of Kirillov-Reshetikhin modules. To date, the only explicit reference to equivariant quantum Kostka numbers in the literature is the unpublished work of Gorbounov and Korff \cite[Corollary 6.27]{GKunpub}, in which they are defined by iterating an equivariant quantum Pieri rule phrased in terms of vicious and osculating walkers. As such, Theorem \ref{thm:main} also provides a simple combinatorial model for performing such calculations in quantum integrable systems.

The cylindric Schur functions of \cite{Postnikov} are examples of affine Schur functions \cite{LamAJM}, also known as the dual $k$-Schur functions which represent Schubert classes in the cohomology of the affine Grassmannian, as one might expect by applying the parabolic Peterson isomorphism \cite{Pet,LamShimPet}.  Analogously, the equivariant version of the cylindric Schur functions obtained by iterating Theorem \ref{thm:main} should give rise to a family of affine double Schur functions.  By \cite[Theorem 27]{LSkDouble}, specializing one set of variables in an affine double Schur function recovers the localizations of the corresponding affine Schubert class at a $T$-fixed point, deepening the geometric meaning of this family of symmetric functions, and simultaneously illuminating the critical role of localization in the proof of Theorem \ref{thm:main}.  The equivariant $k$-Kostka numbers are defined in \cite{LSkDouble} by expanding the affine double Schur functions in terms of the double monomial symmetric functions, or equivalently by iterating the affine equivariant Pieri rule of \cite{LSPieri}, which offers an affine approach to studying equivariant quantum Kostka numbers. The authors also expect the cyclic factorial Schur functions defined in \cite{BMT} to configure into this same framework of affine symmetric functions.

Many additional applications of Theorem \ref{thm:main} are implicit in Postnikov's wide-ranging work \cite{Postnikov}.
There are natural interpretations of the correspondence between cylindric shapes and non-vanishing monomials in the affine nil-Temperley-Lieb algebra, which would give rise to an equivariant generalization of \cite[Proposition 8.5]{Postnikov}, as well as its physical interpretation in \cite[Theorem 10.11]{KorffStroppel}.  This perspective also yields equivariant analogs of the pairwise commuting Dunkl elements in the quadratic algebra which Fomin and Kirillov use to describe the quantum cohomology of the complete flag variety \cite{FomKir}.  Postnikov defines a representation of the symmetric group by its action on the toric Specht module,  and Yoo develops a cylindric version of jeu de taquin to study certain families of these toric Specht modules \cite{Yoo}.  Since the quantum Littlewood-Richardson coefficients conjecturally determine the decomposition of toric Specht modules into irreducible components, an algebraic explanation of the Graham-positivity of the equivariant Littlewood-Richardson coefficients results from this representation-theoretic approach.  See \cite[Section 9]{Postnikov} for further discussion of open problems, each of which now has a natural equivariant analogue.

\subsection{Organization of the paper}

We begin with some background on the equivariant quantum Schubert calculus of the Grassmannian in Section \ref{sec:background}.  In particular, we formally review Postnikov's quantum Pieri rule on cylindric shapes from \cite{Postnikov} as Theorem \ref{thm:post}, and we make all additional terminology appearing in Theorem \ref{thm:main}, our equivariant generalization of Postnikov's quantum Pieri rule, precise in Section \ref{sec:eqQPieri}. The localizations occurring in the equivariant Pieri rule of \cite{HuangLi} are developed in Section \ref{sec:Localizations}, which permits a comparison to the equivariant Littlewood-Richardson coefficients from Theorem \ref{thm:main} by recognizing both as certain specializations of factorial Schur polynomials in Section \ref{sec:FacSchurProof}. The proof of Theorem \ref{thm:main} then follows in Section \ref{sec:proof} by applying the equivariant rim hook rule of \cite{BMT}.

\subsection*{Acknowledgements}

EM thanks Harry Tamvakis for an early discussion encouraging the authors of \cite{BMT} to develop equivariant generalizations of several results in \cite{Postnikov}. Foundational work on this project was carried out by AB, EM, and KT during the Women and Mathematics Program on ``Aspects of Algebraic Geometry'' at the Institute for Advanced Study in May 2015. EM and KT gratefully acknowledge the support of the Max-Planck-Institut f\"ur Mathematik, which hosted two long-term sabbatical visits by EM in 2016 and 2020, in addition to a collaborative visit by KT in May 2016, jointly funded by EM's Simons Collaboration Grant.  A substantial part of this work was also carried out while EM served as a Director's Mathematician in Residence at the Budapest Semesters in Mathematics in Summer 2018, jointly supported by the R\'enyi Alfr\'ed Matematikai Kutat\'oint\'ezet.   DE is especially grateful to the National Science Foundation for fully funding his collaborative work on this project through EM's grant from the Division of Mathematical Sciences. This work also benefited from EM's discussions on related NSF-supported projects with Alex Kane and Chris Gandolfo-Lucia. KT additionally thanks Takeshi Ikeda and Allen Knutson for long-ago conversations about equivariant generalizations of quantum cohomology. Finally, the authors are very grateful to the anonymous referee for an especially careful reading which resulted in many improvements throughout the paper.

%%%%%%%%%%%%%%%%%%%%%%%%%%%%%%%%%%%%%%%%%%%%
%%%%%%%%%%%%%%%%%%%%%%%%%%%%%%%%%%%%%%%%%%%%
%%%%%%%%%%%%%%%%%%%%%%%%%%%%%%%%%%%%%%%%%%%%
%%%%%%%%%%%%%%%%%%%%%%%%%%%%%%%%%%%%%%%%%%%%
%%%%%%%%%%%%%%%%%%%%%%%%%%%%%%%%%%%%%%%%%%%%

\section{Quantum Pieri rules on cylindric shapes}\label{sec:background}

The purpose of this section is to briefly develop the Schubert calculus required to formally state our main theorem. We review the combinatorics of the quantum cohomology of the Grassmannian in Section \ref{sec:QCoh}, and we formulate Postnikov's quantum Pieri rule on cylindric shapes in Section \ref{sec:QPieri}. We pass to the equivariant setting in Section \ref{sec:eqQCoh}, and we present our equivariant quantum Pieri rule on cylindric shapes as Theorem \ref{thm:main} in Section \ref{sec:eqQPieri}. Throughout the paper, we fix a pair of integers $m,n \in \N$.

\subsection{Quantum cohomology of the Grassmannian}\label{sec:QCoh}

Let $G = SL_n$, fix a Borel subgroup $B$ and a split maximal torus $T$, and denote the associated Weyl group by $W = S_n$.  Fix a Cartan subalgebra $\mathfrak{h}$ of $\mathfrak{g} = \operatorname{Lie}(G)$. The choice of $B$ determines a set of positive roots $R^+$, in the sense that $\mathfrak{b} = \operatorname{Lie}(B) = \mathfrak{h} \oplus \left(\bigoplus_{\alpha \in R^+} \mathfrak{g}_\alpha\right)$.  Denote by $\Delta = \{\alpha_i\}_{i=1}^{n-1}$ a basis of simple roots, satisfying that any root in $\bigoplus \Z_{\geq 0} \alpha_i$ is positive with respect to $B$. In this paper, we choose $B$ to be the subgroup of lower-triangular matrices and $T$ the diagonal matrices, in which case $\alpha_i = e_{i+1}-e_i$ where $e_i$ denotes the standard coordinate vector for the lattice $\Z^n$.  There is a bijection between the simple roots $\alpha_i \in \Delta$ and the generators $s_i \in S$ for the Coxeter group $(W,S)$.  The standard parabolic subgroups $P \supseteq B$ are in bijection with subsets $\Delta_P \subseteq \Delta$, and the associated Weyl group $W_P$ is generated by those simple reflections corresponding to the elements of $\Delta_P$.  The set of minimal length coset representatives in the quotient $W/W_P$ is denoted by $W^P$, using the standard length function on $(W,S)$.

The homogeneous variety $G/P$ admits a Bruhat decomposition into \emph{opposite Schubert cells}, given by $G/P = \bigsqcup_{w \in W^P} B^-wP/P$.  Here, $B^-$ denotes the \emph{opposite Borel subgroup}, defined by $B\cap B^- = T$. Following the conventions in this paper, $B^-$ is the subgroup of upper-triangular matrices. For any $w \in W^P$, we denote the corresponding \emph{Schubert variety} by $X_w := \overline{B^-wP/P}$.  The cohomology of $G/P$ admits an additive $\Z$-basis of \emph{Schubert classes} $\sigma_w$ indexed by $w \in W^P$, where $\sigma_w$ denotes the Poincar\'e dual of the fundamental class $[X_w]$. Throughout this paper, we focus exclusively on the \emph{maximal} parabolic subgroups $P$, equivalently $\Delta_P = \Delta \backslash \{ \alpha_m \}$ for a fixed $m \in [n-1]$.

When $P$ is maximal, the complex algebraic variety $G/P$ is the \emph{Grassmannian} $Gr(m,n)$, whose points consist of all $m$-dimensional subspaces of $\C^n$. The cohomology of $Gr(m,n)$ in this case is determined by a Schubert basis which is indexed by the set $P_{mn}$ of \emph{partitions} $\lambda = (\lambda_1, \dots, \lambda_m) \in \Z^m_{\geq 0}$ such that $n-m \geq \lambda_1 \geq \cdots \geq \lambda_m$. Equivalently, $\lambda \in P_{mn}$ determines a \emph{Young diagram} in English notation by drawing $\lambda_i$ left-justified boxes in row $i$, placing row 1 at the top. Thus, $\lambda \in P_{mn}$ if and only if its Young diagram fits inside an $m \times (n-m)$ rectangle. The number of boxes in $\lambda \in P_{mn}$ is denoted by $|\lambda|:= \sum \lambda_i$.

When the homogeneous space $G/P$ is identified with the Grassmannian $Gr(m,n)$, the Schubert variety $X_\lambda$ associated to the partition $\lambda \in P_{mn}$ is determined by the rank conditions
\[X_\lambda := \left\{ V \in Gr(m,n) \ \middle|\  \dim \left(V\cap F_{r_i} \right) \geq i,\ \text{for}\ i \in [m]\right\}, \]
where $r_i:= \lambda_{m-i+1}+i$ and $F_{r_i} := \operatorname{Span}(\textbf{e}_1, \dots, \textbf{e}_{r_i})$ denotes the subspace spanned by the first $r_i$ column vectors of the standard basis for $\C^n$.  The Schubert class $\sigma_\lambda$ is Poincar\'e dual to the fundamental class $[X_\lambda]$.  These Schubert classes form an additive $\Z$-basis for the ring $H^*(Gr(m,n))$, and the product of two Schubert classes $\sigma_\lambda \cdot \sigma_\mu = \sum c_{\lambda,\mu}^{\nu}\sigma_\nu$ expands non-negatively in this basis.  The \emph{Littlewood-Richardson coefficients} $c_{\lambda,\mu}^{\nu} \in \Z_{\geq 0}$ are well understood, with many available combinatorial interpretations; see \cite{Fulton} for an overview.

If we introduce a parameter $q$ of degree $n$, the \emph{quantum cohomology of the Grassmannian} is defined as $QH^*(Gr(m,n)) := \Z[q] \otimes H^*(Gr(m,n))$, and thus admits a $\Z[q]$-basis of Schubert classes indexed by $\lambda \in P_{mn}$, which we again denote by $\sigma_\lambda$. The \emph{quantum Littlewood-Richardson coefficients} are determined by the product of two quantum Schubert classes 
\[\sigma_\lambda * \sigma_\mu = \sum\limits_{\nu,d} c_{\lambda,\mu}^{\nu,d}q^d\sigma_\nu,\] 
where the sum now additionally ranges over $d \in \Z_{\geq 0}$.  The quantum Littlewood-Richardson coefficients are also non-negative integers, and may be computed by various rules, including several variations on the quantum-to-classical principle; see \cite{BCFF,BKT}. The classical Littlewood-Richardson coefficients are recorded as the degree $d=0$ terms in the associated quantum product.

\subsection{A quantum Pieri rule on cylindric shapes}\label{sec:QPieri}

The product in $QH^*(Gr(m,n))$ is completely determined by the \emph{quantum Pieri rule}, a closed combinatorial formula for multiplying by the quantum Schubert classes $\sigma_{(1)^k}$ or $\sigma_{(k)}$; that is, classes indexed by column or row-shaped partitions.  The first quantum Pieri rule for the Grassmannian was obtained by Bertram using Grothendieck's quot scheme  \cite{Bertram}, and Buch then gave a linear-algebraic proof \cite{Buch}. The quantum Pieri rule can also be easily recovered from its classical counterpart using the rim hook algorithm of Bertram, Ciocan-Fontanine, and Fulton \cite{BCFF}. 

An alternative quantum Pieri rule was formulated by Postnikov in terms of cylindric shapes \cite{Postnikov}.  Postnikov's rule more directly captures the quantum-to-affine phenomenon which governs the ring structure of $QH^*(Gr(m,n))$ via the parabolic Peterson isomorphism, as explicated by the third author in joint work with Cookmeyer \cite{CM}.  The purpose of this subsection is to review the version of the quantum Pieri rule from \cite{Postnikov}, which we record as Theorem \ref{thm:post} below.

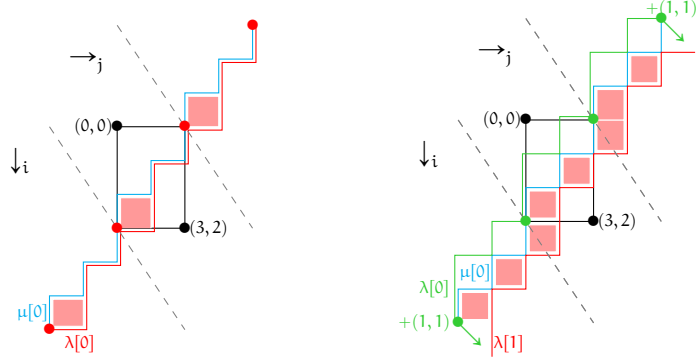
\begin{figure}[h]
\begin{tikzpicture}[scale=.45]

\draw[-]

(10,10) -- (12,10) -- (12,7) -- (10,7) -- cycle
;

\draw[color=cyan]
(14,13) -- (14,12) -- (13,12) -- (13,11) -- (12,11) -- (12,10) -- (12,9) -- (11,9) -- (11,8) -- (10,8) -- (10,7) -- (10,6) -- (9,6) -- (9,5) -- (8,5) -- (8,4)
;

\draw[color=red]
(14.1,13) -- (14.1,11.9) -- (13.1,11.9) -- (13.1,10.9) -- (13.1,9.9) -- (12.1,9.9) -- (12.1,8.9) -- (11.1,8.9) -- (11.1,7.9) -- (11.1,6.9) -- (10.1,6.9) -- (10.1,5.9) -- (9.1,5.9) -- (9.1,4.9) -- (9.1,4.9) -- (9.1,4) -- (8,4)
;

\draw[dashed,color = gray]
(8,10) -- (10,7) -- (12,4)
(10,13) -- (12,10) -- (14,7)
;

\

\path [draw,color = white, fill=red!40] (9,4.1)--  (8.1,4.1) -- (8.1,4.9) -- (9,4.9) -- cycle;
\path [draw,color = white, fill=red!40] (10.1,7) --  (10.1,7.9) -- (11,7.9) -- (11,7);
\path [draw,color = white, fill=red!40] (12.1,10) --  (13,10) -- (13,10.9) -- (12.1,10.9);

\node at (7,9) {$\downarrow$};
\node at (9,12) {$\rightarrow$};
\node[scale = .6] at (9.5,11.8) {$j$};
\node[scale = .6] at (7.3,8.8) {$i$};

\node[scale = .6] at (9.3,10) {$(0,0)$};
\node[scale = .6] at (12.7,7) {$(3,2)$};

\node[color = red] at (8,4) {$\bullet$};
\node[color =  red] at (14,13) {$\bullet$};

\node[color = red] at (12,10) {$\bullet$};
\node[color =  red] at (10,7) {$\bullet$};

\node at (10,10) {$\bullet$};
\node at (12,7) {$\bullet$};

\node[scale = .6, color = cyan] at (7.5,4.5) {$\mu[0]$};
\node[scale = .6, color = red] at (8.9,3.6) {$\lambda[0]$};

\end{tikzpicture}
\quad\quad\quad\quad
\begin{tikzpicture}[scale=.45]

\draw[-]

(10,10) -- (12,10) -- (12,7) -- (10,7) -- cycle
;

\draw[color=cyan]
(14,13) -- (14,12) -- (13,12) -- (13,11) -- (12,11) -- (12,10) -- (12,9) -- (11,9) -- (11,8) -- (10,8) -- (10,7) -- (10,6) -- (9,6) -- (9,5) -- (8,5) -- (8,4)
;

\draw[color=limegreen]
(14,13) -- (13,13) -- (13,12) -- (11.9,12) -- (11.9,11) -- (11.9,10.1) -- (11,10.1) -- (11,9) -- (9.9,9) -- (9.9,8) -- (9.9,7) -- (9,7) -- (9,6) -- (7.9,6) -- (7.9,5) -- (7.9,4);

\draw[color=red]
(15,12) -- (14,12) -- (14,11) -- (13,11) -- (13,10) -- (13,9) -- (12,9) -- (12,8) -- (11,8) -- (11,7) -- (11,6) -- (10,6) -- (10,5) -- (9,5) -- (9,4) -- (9,3);

\draw[dashed,color = gray]
(8,10) -- (10,7) -- (12,4)
(10,13) -- (12,10) -- (14,7)
;

\

\path [draw,color = white, fill=red!40] (8.1,4.1) --  (8.1,4.9) -- (8.9,4.9) -- (8.9,4.1) -- cycle;
\path [draw,color = white, fill=red!40] (9.1,5.1) --  (9.1,5.9) -- (9.9,5.9) -- (9.9,5.1) -- cycle;
\path [draw,color = white, fill=red!40] (10.1,6.1) --  (10.1,6.9) -- (10.9,6.9) -- (10.9,6.1) -- cycle;
\path [draw,color = white, fill=red!40] (10.1,7.1) --  (10.1,7.9) -- (10.9,7.9) -- (10.9,7.1) -- cycle;
\path [draw,color = white, fill=red!40] (11.1,8.1) --  (11.1,8.9) -- (11.9,8.9) -- (11.9,8.1) -- cycle;
\path [draw,color = white, fill=red!40] (12.1,9.1) --  (12.1,10) -- (12.9,10) -- (12.9,9.1) -- cycle;
\path [draw,color = white, fill=red!40] (12.1,10) --  (12.1,10.9) -- (12.9,10.9) -- (12.9,10) -- cycle;
\path [draw,color = white, fill=red!40] (13.1,11.1) --  (13.1,11.9) -- (13.9,11.9) -- (13.9,11.1) -- cycle;

\draw[dashed,color = gray]
(8,10) -- (10,7) -- (12,4)
(10,13) -- (12,10) -- (14,7)
;

\node at (7,9) {$\downarrow$};
\node at (9,12) {$\rightarrow$};
\node[scale = .6] at (9.5,11.8) {$j$};
\node[scale = .6] at (7.3,8.8) {$i$};

\node[scale = .6] at (9.3,10) {$(0,0)$};
\node[scale = .6] at (12.7,7) {$(3,2)$};

\node[color = limegreen] at (8,4) {$\bullet$};
\node[color =  limegreen] at (14,13) {$\bullet$};

\node[color = limegreen] at (12,10) {$\bullet$};
\node[color =  limegreen] at (10,7) {$\bullet$};

\node at (10,10) {$\bullet$};
\node at (12,7) {$\bullet$};

\node[scale = .6, color = cyan] at (8.5,5.5) {$\mu[0]$};
\node[scale = .6, color = limegreen] at (7.3,5) {$\lambda[0]$};
\node[scale = .6, color = red] at (9.5,3.4) {$\lambda[1]$};

\node[color = limegreen] at (8.3,3.7) {$\searrow$};
\node[color = limegreen] at (14.3,12.7) {$\searrow$};
\node[scale = .6,color = limegreen] at (14.3,13.3) {$+(1,1)$};
\node[scale = .6,color = limegreen] at (7,4) {$+(1,1)$};

\end{tikzpicture}

\caption{The cylindric loop $\mu[0]$ for $\mu=(2,1)$ on the cylinder $\mathcal{C}_{35}$  in blue. A cylindric diagram $\lambda/d/\mu$ in red, with $\lambda = (2,1,1)$ and $d=0$ on the left, and $\lambda = (1)$ and $d=1$ on the right. The skew shape on the left is both a vertical and horizontal strip, whereas the skew shape on the right is only a vertical strip.}\label{fig:cylindric}
\end{figure}

Postnikov associates each partition in $P_{mn}$ to a loop on the \emph{cylinder} $\mathcal C_{mn} := \Z^2 / (-m,n-m)\Z$, where the coordinates $(i,j) \in \Z^2$ follow the conventions for indexing matrix entries, meaning that the $i^{\text{th}}$ coordinate increases going down, and the $j^{\text{th}}$ coordinate increases going right; see Figure \ref{fig:cylindric} for an illustration of the cylinder $\mathcal C_{35}$. When referring to a \emph{box} of $\mathcal{C}_{mn}$, we use the cylinder coordinate $(i,j) \in \mathcal C_{mn}$ to refer to the box to its northwest, located between coordinates $(i,j)$ and $(i-1,j-1)$ in $\Z^2$. \emph{Row $k$} of the cylinder consists of the boxes indexed by $(k,j) \in \mathcal{C}_{mn}$, and \emph{column $k$} of the cylinder consists of the boxes indexed by $(i, k) \in \mathcal{C}_{mn}$; note that $\mathcal{C}_{mn}$ has exactly $m$ rows and $n-m$ columns. 

Given a partition $\lambda = (\lambda_1, \dots, \lambda_m) \in P_{mn}$, define a doubly infinite integer sequence $\lambda[0]:= (\dots, \ell_{-1},\ell_0,\ell_1,\dots)$ by setting $\ell_k:= \lambda_k$ for $k \in [m]$, and extending to $\Z$ by defining $\ell_{i+m} := \ell_i - (n-m)$ for all $i \in \Z$. The sequence $\lambda[0]$ determines a \emph{closed loop} on $\mathcal{C}_{mn}$ by plotting $(k,\ell_k) \in \Z^2$ for all $k \in \Z$, and then following the path $(k,\ell_k) \rightarrow (k-1,\ell_k) \rightarrow (k-1,\ell_{k-1})$. The constructed path traces out the boundary of $\lambda$ in the $m \times (n-m)$ rectangle whose northwest corner is the origin.  The closed loop $\lambda[0]$ outlines the boundary of the Young diagram for $\lambda$ periodically in $\Z^2$, in a manner which is invariant under translation by the vector $(-m,n-m)$; see the left-hand illustration in Figure \ref{fig:cylindric} for two examples of closed loops. Given any integer $d \in \Z_{\geq 0}$, the \emph{cylindric loop} $\lambda[d]$ is obtained by shifting  the image of $\lambda[0]$ in $\Z^2$ by $d$ steps in the southeast direction, equivalently by adding the vector $(d,d)$ to every point $(i, \ell_i) \in \Z^2$. For any $j \in \Z$, the $j^{\text{th}}$ entry of the integer sequence corresponding to the cylindric loop $\lambda[d]$ is denoted by $\lambda[d]_j = \ell_{j-d}+d$. The right-hand illustration in Figure \ref{fig:cylindric} depicts $\lambda[0]$ in green and its shifted image $\lambda[1]$ in red.

\begin{definition}
Given any $\lambda,\mu \in P_{mn}$ and any $d \in \Z_{\geq 0}$ such that $\lambda [d]_i \geq \mu_i$ for all $i \in [m]$, the \emph{cylindric diagram} $\lambda/d/\mu$ is represented by those boxes of $\mathcal C_{mn}$ which are located between the two cylindric loops $\lambda[d]$ and $\mu[0]$.
\end{definition}

\noindent In Figure \ref{fig:cylindric}, the red shaded boxes depict two different cylindric diagrams on $\mathcal C_{35}$.

A (weak) \emph{composition} $\eta = (\eta_1, \dots, \eta_m) \in \Z^m_{\geq 0}$ satisfies $n-m \geq \eta_k \geq 0$ for all $k \in [m]$, and has an associated diagram given by drawing $\eta_k$ boxes in row $k$ of the $m \times (n-m)$ rectangle. Given a partition $\lambda \in P_{m,n}$, the composition $\eta \subseteq \lambda$ if and only if $\eta_k \leq \lambda_k$ for all $k \in [m]$.  If $\eta \subseteq \lambda$, the \emph{skew diagram} $\lambda/\eta$ is the set-theoretic difference of the diagrams for $\lambda$ and $\eta$.  As such, all skew Young diagrams also correspond to cylindric diagrams of the form $\lambda/0/\eta$.   We thus typically write $\lambda/\eta$ rather than $\lambda/0/\eta$, and use the term \emph{skew shape} to simultaneously represent the two equivalent perspectives. The left-hand diagram in Figure \ref{fig:cylindric} illustrates this equivalence of skew shapes in both $P_{35}$ and $\mathcal C_{35}$, though for convenience we frequently illustrate only the portion of the skew shape contained in the $m \times (n-m)$ rectangle.

A skew shape is \emph{connected} if for each pair of consecutive rows in the shape, there is also at least one pair of boxes in each row sharing a common edge; neither skew shape depicted in Figure \ref{fig:cylindric} is connected. A \emph{rim hook} is a connected skew shape which does not contain any $2 \times 2$ squares, and a $k$-rim hook for any $k\in \N$ is a rim hook which contains $k$ distinct boxes on $\mathcal C_{mn}$. Importantly, note that the cylindric loop $\lambda[1]$ is equivalently constructed from the Young diagram obtained by adjoining an $n$-rim hook along the boundary of $\lambda \in P_{mn}$ such that the lower leftmost box is located in position $(m+1,1)$; compare the cylindric loops $\lambda[0]$ in green and $\lambda[1]$ in red in the right-hand diagram of Figure \ref{fig:cylindric}. Removing all possible $n$-rim hooks from a Young diagram of any size, in any order, results in the \emph{$n$-core} partition.

\begin{definition}
Let $\lambda \in P_{mn}$, and let $\eta$ be a composition such that $\eta \subseteq \lambda$.  The skew shape $\lambda /\eta$ is a \emph{vertical $r$-strip} (resp.~\emph{horizontal $r$-strip}), denoted $\lambda/\eta = v^r$ (resp.~$\lambda/\eta = h^r$), if its diagram contains exactly $r$ distinct boxes on $\mathcal C_{mn}$, no two in the same row (resp.~ column) of the cylinder $\mathcal{C}_{mn}$.
\end{definition}

\noindent In Figure \ref{fig:cylindric}, the left-hand diagram is both a vertical and horizontal 1-strip, and the right-hand diagram is a vertical 3-strip. However, the right-hand diagram contains two boxes in the first column of the cylinder $\mathcal{C}_{35}$, and is thus not a horizontal 3-strip.

We are now able to formally state the quantum Pieri rule from \cite[Prop.~4.1]{Postnikov}, for which we provide an equivariant generalization in Theorem \ref{thm:main} as our main result.

\begin{theorem}[Quantum Pieri Rule on Cylindric Shapes \cite{Postnikov}]\label{thm:post}
For any integers $1 \leq p \leq m$ and $1 \leq k \leq n-m$ and any partition $\mu \in P_{mn}$, we have
\begin{equation*}
\sigma_{(1)^p} * \sigma_\mu=\sum_{\lambda/d/\mu = v^p} q^d \sigma_\lambda
\qquad 
\text{and}
\qquad
\sigma_{(k)} * \sigma_\mu=\sum_{\lambda/d/\mu = h^k} q^d \sigma_\lambda
\end{equation*}
in $QH^*(Gr(m,n)).$ Since $\lambda/d/\mu$ is a vertical (resp.~horizontal) strip, then $d \in \{0,1\}$. 
\end{theorem}

\subsection{Equivariant quantum cohomology of the Grassmannian}\label{sec:eqQCoh}

There is always a natural left action of the torus $T$ on the homogenous space $G/P$. In the case of maximal $P$, the Schubert varieties $X_\lambda$ are $T$-invariant with codimension $|\lambda|$, and thus also determine a basis of Schubert classes $\sigma_\lambda$  in the \emph{$T$-equivariant cohomology of the Grassmannian} $H_T^*(Gr(m,n))$. In fact, these Schubert classes  also form a basis for $H_T^*(Gr(m,n))$ as a module over $H^*_T(\operatorname{pt})$.  Furthermore, since $G = SL_n$, then the equivariant cohomology of a point is isomorphic to a polynomial ring $H_T^*(\operatorname{pt}) \cong  \Z[t_1, \dots, t_n]$. As such, $H^*_{T}(Gr(m,n))$ admits a $\Z[t_1, \dots, t_n]$-basis of Schubert classes $\sigma_\lambda$ indexed by $\lambda \in P_{mn}$.
 
 The \emph{equivariant Littlewood-Richardson coefficients} are defined by the product of two equivariant Schubert classes
\begin{equation*}
\sigma_\lambda \circ \sigma_\mu = \sum c_{\lambda,\mu}^\nu \sigma_\nu
\end{equation*}
in $H^*_{T}(Gr(m,n))$, where here $c_{\lambda,\mu}^\nu \in \Z[t_1, \dots, t_n]$.   
The equivariant Littlewood-Richardson coefficients satisfy the \emph{Graham-positivity} property \cite{Graham}, which means that $c_{\lambda, \mu}^\nu \in \Z_{\geq 0}[\alpha_1, \dots, \alpha_{n-1}]$. Having chosen $B$ to be lower-triangular in Section \ref{sec:QCoh}, then  $c_{\lambda, \mu}^{\nu}\in \Z_{\geq 0}[t_2-t_1, \dots, t_{n}-t_{n-1}]$. In this context, $\Z_{\geq0}$-linear combinations of the $(t_{i+1}-t_i)$ are referred to as (torus) \emph{weights}. Several different combinatorial formulas for the equivariant Littlewood-Richardson coefficients exist, such as the puzzle rule of Knutson and Tao \cite{KT}, reinterpreted by Zinn-Justin through quantum integrability of tilings \cite{ZinnJustin}, a barred skew tableaux formula obtained independently by Molev \cite{Molev} and Kreiman \cite{Kreiman}, or Thomas and Yong's equivariant adaptation of jeu de taquin \cite{ThomasYong}.  Since we shall work in the equivariant setting for the remainder of the paper, there should be no confusion from adopting the same notation here as in our brief review of the non-equivariant context.

Using the same quantum parameter $q$ as in the non-equivariant setting, the \emph{equivariant quantum cohomology of the Grassmannian} is defined as $QH^*_{T}(Gr(m,n)) := \Z[q] \otimes H^*_{T}(Gr(m,n))$, and is simultaneously a deformation of both $H^*_{T}(Gr(m,n))$ and $QH^*(Gr(m,n))$. The ring $QH^*_{T}(Gr(m,n))$ has an additive $\Z[t_1,\dots, t_n][q]$-basis of Schubert classes $\sigma_\lambda$ for $\lambda \in P_{mn}$. The \emph{equivariant quantum Littlewood-Richardson coefficients} are determined by the product of two Schubert classes
\begin{equation*}
\sigma_\lambda \star \sigma_\mu = \sum\limits_{\nu,d} c_{\lambda,\mu}^{\nu,d}q^d\sigma_\nu
\end{equation*}
in $QH^*_{T}(Gr(m,n))$. The equivariant Littlewood-Richardson coefficients are encoded as the degree $d=0$ terms in the same product, and the corresponding quantum Littlewood-Richardson coefficients can be obtained by setting all $t_i = 0$.  The equivariant quantum Littlewood-Richardson coefficients also display Graham-positivity, as proved by Mihalcea \cite{Mihalcea}, in the sense that $c_{\lambda, \mu}^{\nu,d}\in \Z_{\geq 0}[t_2-t_1, \dots, t_{n}-t_{n-1}]$ using the conventions of this paper.  A positive combinatorial formula for the equivariant Littlewood-Richardson coefficients is obtained by combining Buch and Mihalcea's equivariant quantum-to-classical principle \cite{BuchMihalcea} with Buch's equivariant puzzle rule for two-step flag varieties \cite{Buch2step}.

A simple, signed formula for the equivariant quantum Littlewood-Richardson coefficients is given by combining the equivariant rim hook rule of the first, third, and fourth authors with any available equivariant Littlewood-Richardson rule.  We now briefly review  \cite[Theorem 2.6]{BMT}, in order to apply it in the proof of our main theorem in Section \ref{sec:FacSchurProof}.

\begin{theorem}[Equivariant Rim Hook Rule \cite{BMT}]\label{thm:eqrimhook}
Let $\lambda, \mu \in P_{mn}$, and consider the product of the corresponding classes $\sigma_\lambda \circ \sigma_\mu = \sum c_{\lambda,\mu}^{\gamma}\sigma_\gamma \in H^*_{T}(Gr(m,2n-1))$. Then in $QH^*_{T}(Gr(m,n))$, we have 
\begin{equation*}
\sigma_\lambda \star \sigma_\mu = \sum \varphi \left( c_{\lambda,\mu}^{\gamma}\right)\varphi \left(\sigma_\gamma\right), \quad 
\text{where} 
\quad 
\varphi:
\begin{cases}
t_i \ \mapsto t_{i(\operatorname{mod}n)} \\
\sigma_\gamma \mapsto 
\begin{cases}
\prod_{i=1}^d \left( (-1)^{(\varepsilon_i-m)}q\right)\sigma_{\nu} \quad \text{if}\ \nu \in P_{mn}\\
0 \hskip 1.58in \text{if}\ \nu \notin P_{mn}.
\end{cases}
\end{cases}
\end{equation*}
Here, $\nu$ denotes the $n$-core of $\gamma$, the integer $d$ is the number of $n$-rim hooks removed from $\gamma$ to obtain $\nu$, and $\varepsilon_i$ is the height of the $i^{\text{th}}$ rim hook removed.
\end{theorem}

\noindent Our conventions on Graham-positivity in this paper differ by a sign from those in much of the literature, in order to remain consistent with the torus weight conventions of \cite{BMT}.

The structure of the Grassmannian lends itself to many symmetries, which we exploit to relate the Pieri rule on row and column-shaped partitions in the proof of our main theorem. Given $\lambda \in P_{mn}$, denote by $\lambda' \in P_{n-m,n}$ the \emph{transpose} partition which exchanges the rows and columns.  The isomorphism which identifies an $m$-dimensional subspace of $V= \C^n$ with an $(n-m)$-dimensional subspace of the dual space $V^*$ exchanges the Schubert class $\sigma_\lambda$ for the Schubert class $\sigma_{\lambda'}$  indexed by the transposed partition. The induced isomorphism on equivariant quantum cohomology $QH^*_{T}(Gr(m,n)) \cong QH^*_{T}(Gr(n-m,n))$ also involves an involution on the torus weights. Below we review this \emph{level-rank duality} on equivariant quantum Littlewood-Richardson coefficients, as stated by Gorbounov and Korff in \cite[Corollary 4.14]{GK}; see also \cite[Corollary 6.8]{GKunpub}. 

\begin{theorem}[Level-Rank Duality \cite{GKunpub,GK}]\label{thm:levelrank}
Given $\lambda, \mu, \nu \in \mathcal{P}_{mn}$ and any $d \in \Z_{\geq 0}$, then  
\begin{equation*}
c_{\lambda,\mu}^{\nu,d}(t) = c_{\lambda',\mu'}^{\nu',d}(-w_0 t),
\end{equation*}
where the involution on torus weights is given by the substitution $-w_0: t_i \mapsto -t_{n+1-i}$.
\end{theorem}

\subsection{An equivariant quantum Pieri rule on cylindric shapes}\label{sec:eqQPieri}

In the equivariant setting, Mihalcea showed in \cite{MihalceaTrans} that the product in $QH^*_{T}(Gr(m,n))$  is completely determined by the \emph{equivariant quantum Pieri-Chevalley rule}, a closed combinatorial formula for multiplying by the equivariant quantum Schubert class $\sigma_{(1)}$ corresponding to a single box. The method for determining the algebra structure from this Pieri-Chevalley rule is both recursive and signed, and so more general Graham-positive product formulas have since emerged, including the full equivariant quantum Littlewood-Richardson rule given by combining \cite{BuchMihalcea} and \cite{Buch2step}.

Without appealing to some variation on the quantum-to-classical principle as in \cite{BuchMihalcea} or \cite{BMT}, the most general combinatorial formula for products in $QH^*_{T}(Gr(m,n))$ is the \emph{equivariant quantum Pieri rule} for multiplying by the equivariant quantum Schubert classes $\sigma_{(1)^k}$ or $\sigma_{(k)}$. Huang and Li proved the first equivariant quantum Pieri rule for any type $A$ partial flag variety in Theorem 3.10 of \cite{HuangLi}; see Theorem \ref{T:HuangLi} below, which reviews a special case. Also compare the approach of Gorbounov and Korff using the six-vertex model from statistical mechanics \cite{GKunpub,GK}.  As in the non-equivariant context, the equivariant quantum Pieri rule can also be easily recovered from its classical counterpart using the equivariant rim hook rule of \cite{BMT}.

The goal of this paper is to present a closed combinatorial Graham-positive equivariant quantum Pieri rule in terms of cylindric shapes.  In addition to more directly capturing the quantum-to-affine phenomenon which governs the ring structure of $QH^*_{T}(Gr(m,n))$ via the parabolic Peterson isomorphism of \cite{Pet,LamShimPet}, our rule does not require doing calculations in any related two-step flags or smaller/larger Grassmannians. Theorem \ref{thm:main} below is a direct equivariant generalization of Theorem \ref{thm:post}, in which the equivariant terms enter by recording statistics on boxes which may be ``added back'' to the skew shape, in a manner which we now formalize.

\begin{definition} \label{def:AB}
Let $\lambda, \mu \in P_{mn}$ be such that $\mu \subseteq \lambda$, and let $d \in \Z_{\geq 0}$. Suppose that the cylindric shape $\lambda/ d/ \mu$ is a vertical $r$-strip $v^r$ (resp.~horizontal $r$-strip $h^r$) for some $0 \leq r \leq m$. Fix another integer $r \leq p \leq m$. 
\begin{enumerate}
\item $v^r$ (resp.~$h^r$) is \emph{extendable} to a vertical (resp.~horizontal) $p$-strip if by adding $p-r$ boxes from $\mu$ to $\lambda/d/\mu$, each of which shares a vertical (resp.~horizontal) edge with the boundary of $\lambda/d/\mu$, we can form a vertical $p$-strip $v^p$ (resp.~horizontal $p$-strip $h^p$).  Note that the complement $\lambda \ba v^p$ (resp.~$\lambda \ba h^p$) may no longer be a partition shape, but rather a composition.
\item The $p-r$ boxes from $\mu$ which are added to $\lambda/d/\mu$ in the process of extending $v^r$ to $v^p$ (resp.~$h^r$ to $h^p$) are called \emph{addable boxes}. We typically denote an individual addable box by $\alpha \in v^p \ba v^r$ (resp.~$\alpha \in h^p \ba h^r$), where the original vertical (resp.~horizontal) $r$-strip is always assumed to be of the form $\lambda/d/\mu$.
\item We write $v^r \rightarrow v^p$ (resp.~$h^r \rightarrow h^p$) to denote that the vertical $r$-strip $v^r$ (resp.~horizontal $r$-strip $h^r$) is being \emph{extended} to a vertical $p$-strip $v^p$ (resp.~horizontal $p$-strip $h^p$) via addable boxes, without specific reference to which boxes are being added.  Equivalently, we say that $v^p$ (resp.~$h^p$) is an \emph{extension} of  $v^r$ (resp.~$h^r$).
\end{enumerate}
\end{definition}

We illustrate the concept of extending vertical and horizontal strips via addable boxes in the following example, which we shall then continue to utilize throughout the paper.

\begin{example}
 Consider $\mu = (6,6,6,3,2,0,0) \subset \lambda = (7,6,6,4,2,1,0) \in P_{7,15}$.
As Figure \ref{fig:vstrip} illustrates in the $7\times 8$ rectangle, the skew shape $\lambda/\mu$ consists of the three boxes containing red stars.  Since no two of the boxes containing a $\red{\sstar[1.5]}$ are in the same row or column, this skew shape is both a vertical 3-strip and a horizontal 3-strip on the cylinder $\mathcal C_{7,15}$.

\begin{figure}[h]
\[\begin{tikzpicture}[scale=.45]

\draw[-]

(8,8) -- (16,8) 
(8,7) -- (15,7) 
(8,6) -- (14,6)
(8,5) -- (14,5) 
(8,4) -- (12,4) 
(8,3) -- (10,3) 
(8,2) -- (9,2) 

(8,8) -- (8,1)
(9,8) -- (9,2)
(10,8) -- (10,3)
(11,8) -- (11,4)
(12,8) -- (12,4)
(13,8) -- (13,5)
(14,8) -- (14,5)
(15,8) -- (15,7)
;

\node[color=red] at (14.5,7.5) {$\sstar[2]$};
\node[color=red] at (11.5,4.5) {$\sstar[2]$};
\node[color=red] at (8.5,2.5) {$\sstar[2]$};

\node at (13.5,6.5) {$\green{\sbullet[2]}$};
\node at (13.5,5.5) {$\green{\sbullet[2]}$};

\draw[-]

(18,8) -- (26,8) 
(18,7) -- (25,7) 
(18,6) -- (24,6)
(18,5) -- (24,5) 
(18,4) -- (22,4) 
(18,3) -- (20,3) 
(18,2) -- (19,2) 

(18,8) -- (18,1)
(19,8) -- (19,2)
(20,8) -- (20,3)
(21,8) -- (21,4)
(22,8) -- (22,4)
(23,8) -- (23,5)
(24,8) -- (24,5)
(25,8) -- (25,7)
;

\node[color=red] at (24.5,7.5) {$\sstar[2]$};
\node[color=red] at (21.5,4.5) {$\sstar[2]$};
\node[color=red] at (18.5,2.5) {$\sstar[2]$};

\node at (23.5,6.5) {$\green{\sbullet[2]}$};
\node at (19.5,3.5) {$\green{\sbullet[2]}$};

\draw[-]

(28,8) -- (36,8) 
(28,7) -- (35,7) 
(28,6) -- (34,6)
(28,5) -- (34,5) 
(28,4) -- (32,4) 
(28,3) -- (30,3) 
(28,2) -- (29,2) 

(28,8) -- (28,1)
(29,8) -- (29,2)
(30,8) -- (30,3)
(31,8) -- (31,4)
(32,8) -- (32,4)
(33,8) -- (33,5)
(34,8) -- (34,5)
(35,8) -- (35,7)
;

\node[color=red] at (34.5,7.5) {$\sstar[2]$};
\node[color=red] at (31.5,4.5) {$\sstar[2]$};
\node[color=red] at (28.5,2.5) {$\sstar[2]$};

\node at (33.5,5.5) {$\green{\sbullet[2]}$};
\node at (29.5,3.5) {$\green{\sbullet[2]}$};

\end{tikzpicture}\]
\caption{Extending a vertical 3-strip to a vertical 5-strip via addable boxes.}\label{fig:vstrip}
\end{figure}
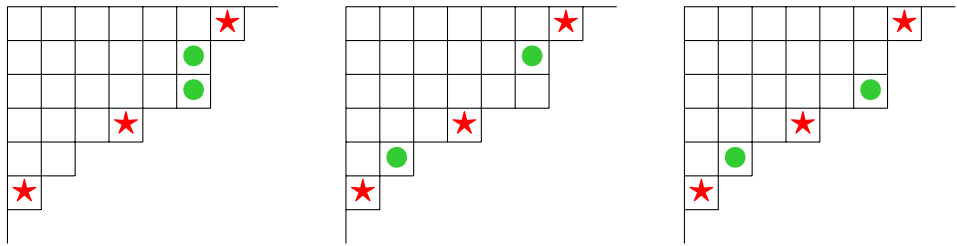

If we now fix $p=5$, then there are 3 different ways to extend $v^3$ to a vertical 5-strip, shown in Figure \ref{fig:vstrip}. In each diagram, the $5-3=2$ boxes of $\mu$ containing green dots are addable, since each box with a $\green{\sbullet[2]}$ shares a vertical edge with the boundary of the original skew shape $\lambda/\mu$, and the union of the boxes with $\red{\bigstar}$ and $\green{\sbullet[2]}$ forms a  vertical 5-strip. For the 5-strip $v^5$ depicted in the middle diagram, note that the complement $\lambda \ba v^5$ is not a partition shape, whereas the other two are.  Each of the three configurations in Figure \ref{fig:vstrip} would contribute to a sum indexed by the extension $v^3 \rightarrow v^5$.

As noted above, the cylindric shape $\lambda/\mu$ is also a horizontal 3-strip $h^3$. However, only the right-hand arrangement of addable boxes in Figure \ref{fig:vstrip} represents a valid extension of $h^3$ to a horizontal 5-strip, since each addable box $\green{\sbullet[2]}$ is required to share a horizontal edge with the boundary of $\lambda/\mu$.  The other 5 configurations which would contribute to a sum indexed by $h^3 \rightarrow h^5$ are not pictured.
\end{example}

We now define several natural statistics on an addable box, determined by its location within the extension. We first enumerate the edges of the partition $\mu \in P_{mn}$, equivalently the closed loop  $\mu[0]$ on $\mathcal C_{mn}$.

\begin{definition}\label{def:upside}
Given $\mu \in P_{mn}$, starting in the lower left-hand corner of the $m \times (n-m)$ rectangle, number all the edges of the path that traces out the boundary of $\mu$. 
\begin{enumerate}
\item The \emph{up-steps of $\mu$}, denoted by $U(\mu)$, are the numbers indexing the vertical edges of $\mu$.
\item The \emph{side-steps of $\mu$}, denoted by $S(\mu)$, are the numbers indexing the horizontal edges of $\mu$. 
\end{enumerate}
Note that $U(\mu) \sqcup S(\mu) = [n]$. 
\end{definition}

The required statistics depend on whether an addable box is viewed as part of a vertical or horizontal strip, and so the following definition is presented as two corresponding sets of statistics.

\begin{definition}\label{def:addboxstats}
 Given a vertical strip $\lambda/d/\mu = v^r$, an extension $v^r \rightarrow v^p$, and an addable box $\alpha \in v^p\backslash v^r$, define:
\begin{enumerate}
\item[(1v)] The \emph{up-step} of $\alpha$, denoted $u(\alpha)$, is the index of the vertical edge of the box $\alpha$ from $U(\mu)$. 
\item[(2v)] The \emph{row number} of $\alpha$, denoted $r(\alpha)$, enumerates the row containing $\alpha$, where we count rows from the bottom of the $m \times (n-m)$ rectangle, equivalently from row $m$ of $\mathcal{C}_{mn}$.
\item[(3v)] The \emph{number of boxes below} $\alpha$, denoted $b(\alpha)$, is defined to be the total number of boxes in the extension $v^p$ lying in rows strictly below $\alpha$ in the $m \times (n-m)$ rectangle, equivalently boxes in $v^p$ with row index in $\mathcal{C}_{mn}$ strictly greater than $r(\alpha)$.
\end{enumerate}
\noindent   Given a horiztonal strip $\lambda/d/\mu = h^r$, an extension $h^r \rightarrow h^p$, and an addable box $\alpha \in h^p\backslash h^r$, define:
\begin{enumerate}
\item[(1h)] The \emph{side-step} of $\alpha$, denoted $s(\alpha)$, is the index of the horizontal edge of the box $\alpha$ from $S(\mu)$. 
\item[(2h)] The \emph{column number} of $\alpha$, denoted $c(\alpha)$, enumerates the column containing $\alpha$, where we count columns from the right of the $m \times (n-m)$ rectangle, equivalently from column $n-m$ of $\mathcal{C}_{mn}$.
\item[(3h)] The \emph{number of boxes to the right} of $\alpha$, denoted $rt(\alpha)$, is defined to be the total number of boxes in the extension $h^p$ lying in columns strictly to the right of $\alpha$ in the $m \times (n-m)$ rectangle, equivalently boxes in $h^p$ with column index in $\mathcal{C}_{mn}$ strictly greater than $c(\alpha)$.
\end{enumerate}
\end{definition}

We can now formally define the weight of an addable box.

\begin{definition}\label{def:wtaddbox}
Let $\lambda, \mu \in P_{mn}$ be such that $\mu \subseteq \lambda$, and let $d \in \Z_{\geq 0}$.
\begin{enumerate}
\item If the cylindric shape $\lambda/ d/ \mu$ is a vertical $r$-strip $v^r$, we define the \emph{weight of an addable box} $\alpha \in v^p \ba v^r$ to be
\begin{equation*}
\wtv(\alpha) := t_{u(\alpha)}-t_{r(\alpha)-b(\alpha)}.
\end{equation*}
\item If the cylindric shape $\lambda/ d/ \mu$ is a horizontal $r$-strip $h^r$, we define the \emph{weight of an addable box} $\alpha \in h^p \ba h^r$ to be
\begin{equation*}
\wth(\alpha) := t_{n+1-(c(\alpha)-rt(\alpha))}-t_{s(\alpha)}.
\end{equation*}
\end{enumerate}
When the original skew shape $\lambda/d/\mu=v^r$ (resp.~$\lambda/d/\mu=h^r$) is not clear from context, we write $\wtv_{\mu}(\alpha)$ (resp.~$\wth_{\mu}(\alpha)$) to indicate that the box is being added from $\mu$. 
\end{definition}

\begin{remark}\label{rmk:rectreduction}
Since addable boxes belong to the interior diagram of the skew shape $\lambda/d/\mu$, each of the statistics in Definition \ref{def:addboxstats} is equivalent whether we consider the partition $\mu \in P_{mn}$ in the rectangle or the closed loop $\mu[0] \subset \mathcal C_{mn}$ on the cylinder.  The weight of an addable box is thus independent of whether we consider the cylindric shape $\lambda/d/\mu$ or the equivalent skew diagram which fits inside the $m \times (n-m)$ rectangle.  We thus refer to the weight of an addable box in a skew shape to simultaneously invoke both meanings.
\end{remark}

To illustrate this definition, we calculate the weight of several addable boxes from Figure \ref{fig:vstrip}.

\begin{example}\label{ex:ABweight}
Recall from Figure \ref{fig:vstrip} that $\lambda=(7,6,6,4,2,1,0) \in P_{7,15}$. Tracing the boundary of $\mu$ in the $7\times 8$ rectangle, we obtain \[U(\mu) = \{1,2,5,7,11,12,13\} \quad \text{and} \quad S(\mu) = \{3,4,6,8,9,10,14,15\}.\]
\begin{enumerate}
\item We first focus on the vertical strip depicted in the left-hand diagram, for which we calculate the weight of each addable box containing a $\green{\sbullet[2]}$. Denote by $\alpha_T$ the top addable box, and by $\alpha_B$ the bottom addable box. Compute using Definition \ref{def:addboxstats} that
\begin{align*}
u(\alpha_T) & = 12 \quad r(\alpha_T) = 6 \quad b(\alpha_T) = 3\\
u(\alpha_B) & = 11 \quad r(\alpha_B) = 5 \quad b(\alpha_B) = 2.
\end{align*} 
Using Definition \ref{def:wtaddbox}, we then compute the weight of each of the two addable boxes occurring in the left-hand diagram from Figure \ref{fig:vstrip} to be
\[ \wtv(\alpha_T) = t_{12}-t_{6-3} = t_{12}-t_3 \quad \text{and} \quad\wtv(\alpha_B) = t_{11}-t_{5-2} = t_{11}-t_3. \]

\item For the purpose of illustrating these calculations in the case of a horizontal strip, we now choose to view the right-hand diagram in Figure \ref{fig:vstrip} as a horizontal strip. Denote by $\alpha_R$ the right addable box, and by $\alpha_L$ the left addable box. Using Definition \ref{def:addboxstats},
\begin{align*}
s(\alpha_R) & = 10 \quad c(\alpha_R) = 3 \quad rt(\alpha_R) = 1\\
s(\alpha_L) & = \ 4\,\, \quad c(\alpha_L) = 7 \quad rt(\alpha_L) = 3.
\end{align*} 
By Definition \ref{def:wtaddbox}, we then compute the weight of each of the two addable boxes occurring in the right-hand diagram from Figure \ref{fig:vstrip} to be
\[ \wth(\alpha_R) =  t_{16-(3-1)}-t_{10}= t_{14}-t_{10} \quad \text{and} \quad\wth(\alpha_L) = t_{16-(7-3)}-t_{4} = t_{12}-t_4. \]
\end{enumerate}
\end{example}

 We are now able to formally state the main theorem in this paper, which provides an equivariant generalization of the quantum Pieri rule on cylindric shapes from \cite{Postnikov}, as well as closed combinatorial formulas for the equivariant quantum Pieri expansions provided in \cite{GKunpub,GK}.  For ease of reference, we repeat the theorem statement from the introduction here.

\EqQPieri*

\noindent Given an addable box $\alpha$ in a vertical strip, since $u(\alpha) >r(\alpha) > r(\alpha)-b(\alpha)$, then clearly $\wtv(\alpha) \in \Z_{\geq 0}[t_2-t_1,\dots, t_n-t_{n-1}]$, and similarly for horizontal strips.  In particular, this equivariant quantum Pieri rule is manifestly Graham-positive, using the conventions of this paper.  Moreover, setting all $t_i=0$ immediately recovers Postnikov's quantum Pieri rule from Theorem \ref{thm:post}.

We now provide several examples illustrating how to use Theorem \ref{thm:main} to efficiently compute equivariant quantum Littlewood-Richardson coefficients, directly on the skew shape $\lambda / \mu$ or $\lambda/1/\mu$.

\begin{example}\label{ex:thmex}
Fix $p=5$, and let $\mu = (6,6,6,3,2,0,0) \in P_{7,15}$. Since we have already seen that the cylindric shape $\lambda/\mu$ for the partition $\lambda = (7,6,6,4,2,1,0)$ is a vertical 3-strip $v^3$, we may use Theorem \ref{thm:main} to calculate the equivariant quantum Littlewood-Richardson coefficient $c_{(1)^5, \mu}^{\lambda,0}$.  Figure \ref{fig:vstrip} shows the three possible ways to extend $v^3 \rightarrow v^5$. By Theorem \ref{thm:main}, the left-hand configuration of addable boxes depicted in Figure \ref{fig:vstrip} contributes the product $(t_{12}-t_3)(t_{11}-t_3)$ of the weights calculated in Example \ref{ex:ABweight}.  In like manner, by Theorem \ref{thm:main}, in $QH^*_{T}(Gr(7,15))$ we have
\[ c_{(1)^5, \mu}^{\lambda,0} =  (t_{12}-t_3)(t_{11}-t_3) + (t_{12}-t_3)(t_{5}-t_2) + (t_{11}-t_2)(t_{5}-t_2),\]
recording the contributions from the remaining two configurations of addable boxes from left to right, where the weight of the top addable box is recorded on the left within each product.  
\end{example}

For comparison, we now use the equivariant quantum Pieri rule on cylindric shapes to calculate a full product in a smaller Grassmannian. We  highlight the fact that the entire calculation in Example \ref{ex:thmexfull} is performed quickly and directly on the four skew shapes  in Figure \ref{fig:fullprod}, without passing to any other Grassmannian or two-step flag variety.

\begin{example}\label{ex:thmexfull}
We now use Theorem \ref{thm:main} to compute the product $\sigma_{\tableau[pcY]{\\ \\ \\ }} \star \sigma_{\tableau[pcY]{& \\ \\ }} \in QH^*_T(Gr(3,5))$. Since $p=3$, we first identify all partitions $\lambda \in P_{35}$ such that $\lambda/d/\mu$ is a vertical $r$-strip for some $0 \leq r \leq 3$ and $d \in \{0,1\}$, and such that $\lambda/d/\mu$ can be extended to a vertical $3$-strip.  A quick inspection reveals that $\lambda \in \left\{ \tableau[pcY]{& \\ \\ \\ },\  \tableau[pcY]{& \\ & \\ \\ }\,, \sbullet[.5]\,,\  \tableau[pcY]{\\}\right\}$, and the 4 respective cylindric diagrams are depicted left to right by the red shaded boxes in Figure \ref{fig:fullprod}; refer to Figure \ref{fig:cylindric} and the surrounding discussion for more details on constructing the relevant cylindric loops.

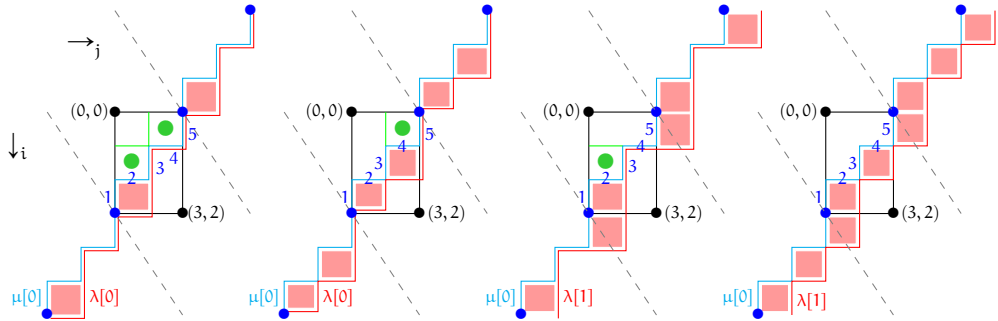
\begin{figure}[h]

\[
\begin{tikzpicture}[scale=.45]

\draw[-]

(10,10) -- (12,10) -- (12,7) -- (10,7) -- cycle
;

\draw[color=cyan]
(14,13) -- (14,12) -- (13,12) -- (13,11) -- (12,11) -- (12,10) -- (12,9) -- (11,9) -- (11,8) -- (10,8) -- (10,7) -- (10,6) -- (9,6) -- (9,5) -- (8,5) -- (8,4)
;

\draw[color=red]
(14.1,12.9) -- (14.1,11.9) -- (13.1,11.9) -- (13.1,10.9) -- (13.1,9.9) -- (12.1,9.9) -- (12.1,8.9) -- (11.1,8.9) -- (11.1,7.9) -- (11.1,6.9) -- (10.1,6.9) -- (10.1,5.9) -- (9.1,5.9) -- (9.1,4.9) -- (9.1,3.9) -- (8.1,3.9)
;

\draw[color=green]
(10,9) -- (11,9) -- (11,10);

\node[color = limegreen] at (10.5,8.5) {$\sbullet[1.5]$};
\node[color = limegreen] at (11.5,9.5) {$\sbullet[1.5]$};

\draw[dashed,color = gray]
(8,10) -- (10,7) -- (12,4)
(10,13) -- (12,10) -- (14,7);

\node[scale = .6, color = cyan] at (7.4,4.5) {$\mu[0]$};
\node[scale = .6, color = red] at (9.7,4.5) {$\lambda[0]$};

\path[draw,color = white, fill = red!40] (8.1,4) -- (8.1,4.9) -- (9,4.9) -- (9,4) -- cycle;
\path[draw,color = white, fill = red!40] (10.1,7.1) -- (10.1,7.9) -- (11,7.9) -- (11,7.1) -- cycle;
\path[draw,color = white, fill = red!40] (12.1,10) -- (12.1,10.9) -- (13,10.9) -- (13,10) -- cycle;

\node[scale = .6,color = blue] at (9.8,7.5) {$1$};
\node[scale = .6,color = blue] at (10.5,8.05) {$2$};
\node[scale = .6,color = blue] at (11.35,8.35) {$3$};
\node[scale = .6,color = blue] at (11.75,8.65) {$4$};
\node[scale = .6,color = blue] at (12.3,9.35) {$5$};

\node at (7,9) {$\downarrow$};
\node at (9,12) {$\rightarrow$};
\node[scale = .6] at (9.5,11.8) {$j$};
\node[scale = .6] at (7.3,8.8) {$i$};

\node[scale = .6] at (9.25,10) {$(0,0)$};
\node[scale = .6] at (12.7,7) {$(3,2)$};

\node[color = blue] at (8,4) {$\bullet$};
\node[color =  blue] at (14,13) {$\bullet$};

\node[color = blue] at (12,10) {$\bullet$};
\node[color =  blue] at (10,7) {$\bullet$};

\node at (10,10) {$\bullet$};
\node at (12,7) {$\bullet$};

%%%%%%%%%%%%%%%%%%%%%%%%%%%%%%%%%%%%%%%%%%%%

\draw[-]

(17,10) -- (19,10) -- (19,7) -- (17,7) -- cycle
;

\draw[color=cyan]
(21,13) -- (21,12) -- (20,12) -- (20,11) -- (19,11) -- (19,10) -- (19,9) -- (18,9) -- (18,8) -- (17,8) -- (17,7) -- (17,6) -- (16,6) -- (16,5) -- (15,5) -- (15,4)
;

\draw[color=red]
(21.1,13.1) -- (21.1,12.1) -- (21.1,11) -- (20,11) -- (20,10.1) -- (19.1,10.1) -- (19.1,9.1) -- (19.1,8) -- (18,8) -- (18,7.1) -- (17.1,7.1) -- (17.1,6.1) -- (17.1,5) -- (16,5) -- (16,4.1) -- (15.1,4.1)
;

\draw[color = green]
(18,10) -- (18,9);

\draw[dashed,color = gray]
(15,10) -- (17,7) -- (19,4)
(17,13) -- (19,10) -- (21,7)
;

\node[scale = .6] at (16.2,10) {$(0,0)$};
\node[scale = .6] at (19.8,7) {$(3,2)$};

\node[color = blue] at (15,4) {$\bullet$};
\node[color =  blue] at (21,13) {$\bullet$};

\node[color = blue] at (19,10) {$\bullet$};
\node[color =  blue] at (17,7) {$\bullet$};

\node at (17,10) {$\bullet$};
\node at (19,7) {$\bullet$};

\node[color = limegreen] at (18.5,9.5) {$\sbullet[1.5]$};

\path[draw,color = white, fill = red!40] (15.1,4.2) -- (15.1,4.9) -- (15.9,4.9) -- (15.9,4.2) -- cycle;
\path[draw,color = white, fill = red!40] (16.1,5.1) -- (16.1,5.9) -- (17,5.9) -- (17,5.1) -- cycle;
\path[draw,color = white, fill = red!40] (17.1,7.2) -- (17.1,7.9) -- (17.9,7.9) -- (17.9,7.2) -- cycle;
\path[draw,color = white, fill = red!40] (17.1,7.2) -- (17.1,7.9) -- (17.9,7.9) -- (17.9,7.2) -- cycle;
\path[draw,color = white, fill = red!40] (18.1,8.1) -- (18.1,8.9) -- (18.9,8.9) -- (18.9,8.1) -- cycle;
\path[draw,color = white, fill = red!40] (19.1,10.15) -- (19.1,10.9) -- (19.9,10.9) -- (19.9,10.15) -- cycle;
\path[draw,color = white, fill = red!40] (20.1,11.1) -- (20.1,11.9) -- (21,11.9) -- (21,11.1) -- cycle;

\node[scale = .6, color = blue] at (16.8,7.5) {$1$};
\node[scale = .6, color = blue] at (17.5,8.05) {$2$};
\node[scale = .6, color = blue] at (17.8,8.5) {$3$};
\node[scale = .6, color = blue] at (18.5,9) {$4$};
\node[scale = .6, color = blue] at (19.3,9.35) {$5$};

\node[scale = .6,color = cyan] at (14.4,4.5) {$\mu[0]$};
\node[scale = .6,color = red] at (16.6,4.5) {$\lambda[0]$};

%%%%%%%%%%%%%%%%%%%%%%%%%%%%%%%%%%%%%%%%%%%%

\draw[-]

(24,10) -- (26,10) -- (26,7) -- (24,7) -- cycle
;

\draw[color=cyan]
(28,13) -- (28,12) -- (27,12) -- (27,11) -- (26,11) -- (26,10) -- (26,9) -- (25,9) -- (25,8) -- (24,8) -- (24,7) -- (24,6) -- (23,6) -- (23,5) -- (22,5) -- (22,4)
;

\draw[color=red]
(29.1,13) -- (29.1,11.9) -- (28.1,11.9) -- (27.1,11.9) -- (27.1,10.9) -- (27.1,9.9) -- (27.1,8.9) -- (26.1,8.9) -- (25.1,8.9) -- (25.1,7.9) -- (25.1,6.9) -- (25.1,5.9) -- (24.1,5.9) -- (23.1,5.9) -- (23.1,4.9) -- (23.1,3.9)
;

\draw[color = green]
(25,9) -- (24,9)
;

\path[draw,color = white, fill = red!40] (22.1,4.1) -- (22.1,4.9) -- (23,4.9) -- (23,4.1) -- cycle;
\path[draw,color = white, fill = red!40] (24.1,6) -- (24.1,6.9) -- (25,6.9) -- (25,6) -- cycle;
\path[draw,color = white, fill = red!40] (24.1,7.1) -- (24.1,7.9) -- (25,7.9) -- (25,7.1) -- cycle;
\path[draw,color = white, fill = red!40] (26.1,9) -- (26.1,9.95) -- (27,9.95) -- (27,9) -- cycle;
\path[draw,color = white, fill = red!40] (26.1,10) -- (26.1,10.9) -- (27,10.9) -- (27,10) -- cycle;
\path[draw,color = white, fill = red!40] (28.1,12) -- (28.1,12.9) -- (29,12.9) -- (29,12) -- cycle;

\node[scale = .6] at (23.2,10) {$(0,0)$};
\node[scale = .6] at (26.8,7) {$(3,2)$};

\node[color = blue] at (22,4) {$\bullet$};
\node[color =  blue] at (28,13) {$\bullet$};

\node[color = blue] at (26,10) {$\bullet$};
\node[color =  blue] at (24,7) {$\bullet$};

\node at (24,10) {$\bullet$};
\node at (26,7) {$\bullet$};

\node[color = limegreen] at (24.5,8.5) {$\sbullet[1.5]$};

\node[color = cyan,scale = .6] at (21.4,4.5) {$\mu[0]$};
\node[color = red,scale = .6] at (23.7,4.5) {$\lambda[1]$};

\node[color = blue,scale = .6] at (23.8,7.5) {$1$};
\node[color = blue,scale = .6] at (24.5,8.05) {$2$};
\node[color = blue,scale = .6] at (25.3,8.4) {$3$};
\node[color = blue,scale = .6] at (25.55,9) {$4$};
\node[color = blue,scale = .6] at (25.8,9.5) {$5$};

\draw[dashed,color = gray]
(22,10) -- (24,7) -- (26,4)
(24,13) -- (26,10) -- (28,7);
 
%%%%%%%%%%%%%%%%%%%%%%%%%%%%%%%%%%%%%%%%%%%%

\draw[-]
(31,10) -- (33,10) -- (33,7) -- (31,7) -- cycle
;

\draw[color=cyan]
(35,13) -- (35,12) -- (34,12) -- (34,11) -- (33,11) -- (33,10) -- (33,9) -- (32,9) -- (32,8) -- (31,8) -- (31,7) -- (31,6) -- (30,6) -- (30,5) -- (29,5) -- (29,4)
;

\draw[color=red]
(36,13) -- (36,12) -- (35,12) -- (35,11) -- (34,11) -- (34,10) -- (34,9) -- (33,9) -- (33,8) -- (32,8) -- (32,7) -- (32,6) -- (31,6) -- (31,5) -- (30,5) -- (30,4)
;

\node[scale = .6] at (30.2,10) {$(0,0)$};
\node[scale = .6] at (33.8,7) {$(3,2)$};

\node at (31,10) {$\bullet$};
\node at (33,7) {$\bullet$};

\node[scale = .6,color = cyan] at (28.4,4.5) {$\mu[0]$};
\node[scale = .6,color = red] at (30.6,4.4) {$\lambda[1]$};

\path[draw,color = white, fill = red!40] (29.1,4.1) -- (29.1,4.9) -- (29.9,4.9) -- (29.9,4.1) -- cycle;
\path[draw,color = white, fill = red!40] (30.1,5.1) -- (30.1,5.9) -- (30.9,5.9) -- (30.9,5.1) -- cycle;
\path[draw,color = white, fill = red!40] (31.1,6.1) -- (31.1,6.9) -- (31.9,6.9) -- (31.9,6.1) -- cycle;
\path[draw,color = white, fill = red!40] (31.1,7.1) -- (31.1,7.9) -- (31.9,7.9) -- (31.9,7.1) -- cycle;
\path[draw,color = white, fill = red!40] (32.1,8.1) -- (32.1,8.9) -- (32.9,8.9) -- (32.9,8.1) -- cycle;
\path[draw,color = white, fill = red!40] (33.1,9.1) -- (33.1,9.95) -- (33.9,9.95) -- (33.9,9.1) -- cycle;
\path[draw,color = white, fill = red!40] (33.1,10.05) -- (33.1,10.9) -- (33.9,10.9) -- (33.9,10.05) -- cycle;
\path[draw,color = white, fill = red!40] (34.1,11.1) -- (34.1,11.9) -- (34.9,11.9) -- (34.9,11.1) -- cycle;
\path[draw,color = white, fill = red!40] (35.1,12.1) -- (35.1,12.9) -- (35.9,12.9) -- (35.9,12.1) -- cycle;

\node[color = blue] at (29,4) {$\bullet$};
\node[color =  blue] at (35,13) {$\bullet$};

\node[color = blue] at (33,10) {$\bullet$};
\node[color =  blue] at (31,7) {$\bullet$};

\node[scale = .6,color = blue] at (30.8,7.5) {$1$};
\node[scale = .6,color = blue] at (31.5,8.05) {$2$};
\node[scale = .6,color = blue] at (31.8,8.5) {$3$};
\node[scale = .6,color = blue] at (32.5,9) {$4$};
\node[scale = .6,color = blue] at (32.8,9.5) {$5$};

\draw[dashed,color = gray]
(29,10) -- (31,7) -- (33,4)
(31,13) -- (33,10) -- (35,7);

%%%%%%%%%%%%%%%%%%%%%%%%%%%%%%%%%%%%%%%%%%%%

;

\end{tikzpicture}
\]

\caption{Using Theorem \ref{thm:main} to calculate $\sigma_{(1)^3} \star \sigma_{(2,1)} \in QH^*_T(Gr(3,5))$.}\label{fig:fullprod}
\end{figure}

In this example, for each $\lambda/d/\mu = v^r$, there is a unique extension $v^r \rightarrow v^3$, and the addable boxes in each extension are indicated by a green $\green{\sbullet[1.5]}$ in Figure \ref{fig:fullprod}.  Calculating the weight of each addable box using Definition \ref{def:wtaddbox}, by Theorem \ref{thm:main} we then immediately obtain 
\[ \sigma_{\tableau[pcY]{\\ \\ \\ }} \star \sigma_{\tableau[pcY]{& \\ \\ }} = (t_5-t_1)(t_3-t_1)\sigma_{\tableau[pcY]{& \\ \\ \\ }} + (t_5-t_1)\sigma_{\tableau[pcY]{& \\ & \\ \\ }} + q(t_3-t_1) \sigma_{\sbullet[.5]} +q \sigma_{\tableau[pcY]{\\}} \in QH^*_T(Gr(3,5)),\]
where we record the product in the same order as the corresponding diagram in Figure \ref{fig:fullprod}.
\end{example}

%%%%%%%%%%%%%%%%%%%%%%%%%%%%%%%%%%%%%%%%%%%%
%%%%%%%%%%%%%%%%%%%%%%%%%%%%%%%%%%%%%%%%%%%%
%%%%%%%%%%%%%%%%%%%%%%%%%%%%%%%%%%%%%%%%%%%%
%%%%%%%%%%%%%%%%%%%%%%%%%%%%%%%%%%%%%%%%%%%%
%%%%%%%%%%%%%%%%%%%%%%%%%%%%%%%%%%%%%%%%%%%%

\section{Equivariant Pieri rules via localization}\label{sec:Localizations}

The goal of this section is to relate the weights which arise in the equivariant quantum Pieri rule on cylindric shapes of Theorem \ref{thm:main} with those appearing as localizations in the corresponding formula of Huang and Li reviewed in Section \ref{sec:PieriLoc} as Theorem \ref{T:HuangLi}. We provide the required background on the localization map in Section \ref{sec:loc}, and then relate the localizations appearing in Theorem \ref{T:HuangLi} to the weights occurring in Theorem \ref{thm:main} in Section \ref{sec:ABcutjoin}.

\subsection{The localization map}\label{sec:loc}

In \cite{KostantKumar}, Kostant and Kumar define a family of functions which determines the ring structure for the equivariant cohomology of any Kac-Moody flag variety. In the special case of the Grassmannian, the Schubert class $[X_\gamma] \in H^*_{T}(Gr(m,n))$ for any partition
 $\gamma \in P_{mn}$ is identified with a function $\xi^\gamma: P_{mn} \rightarrow \Z[t_1, \dots, t_n]$. When evaluated at an element $\eta \in P_{mn}$, the polynomial $\xi^\gamma(\eta)$ is homogeneous of degree $|\gamma|$. Moreover, the polynomials $\xi^\gamma(\eta)$ satisfy the Graham-positivity property, meaning that the function values $\xi^\gamma(\eta) \in \Z_{\geq 0}[t_2-t_1, \dots, t_{n}-t_{n-1}]$, using the conventions in this paper; refer to Sections \ref{sec:QCoh} and \ref{sec:eqQCoh} for more details.

We briefly review the precise relationship between the function $\xi^\gamma$ and the Schubert class $[X_\gamma]$, referring the reader to \cite[Section 11.3]{Kumar} for more details.  The inclusion of the $T$-fixed points in $Gr(m,n)$, themselves also indexed by partitions $\eta \in P_{mn}$, induces an injection 
\[ H^*_{T}(Gr(m,n)) \hookrightarrow H^*_{T}(Gr(m,n)^{T}) \cong \bigoplus_{\eta \in P_{mn}} \Z[t_1, \dots, t_n]. \]
By definition, the induced map sends the Schubert class $[X_\gamma]$ to the collection of localizations $[X_\gamma] |_\eta$ at all $T$-fixed points indexed by $\eta \in P_{mn}$.  In turn, under the above isomorphism, the localization $[X_\gamma] |_\eta$ is identified with the polynomial $\xi^\gamma(\eta)$.
Altogether, this localization map identifies each $T$-equivariant Schubert class with a polynomial-valued function via $[X_\gamma] \mapsto \xi^\gamma$.  As such, the ring structure of $H^*_{T}(Gr(m,n))$ with respect to the Schubert basis is completely determined by the pointwise product of the family of functions $\xi^\gamma$. Moreover, the image of $H^*_{T}(Gr(m,n))$ in this polynomial ring has a simple characterization due to  Goresky, Kottwitz, and MacPherson \cite{GKM}, which is the starting point for GKM theory; see the survey by Tymoczko  \cite{TymGKM}. If the ambient Grassmannian is not clear from context (e.g.~if it differs from $Gr(m,n)$), we shall use a subscript $\xi^\gamma_k(\eta)$ to indicate that $\gamma,\eta \in P_{kn}$ instead.

The functions $\xi^\gamma$ are uniquely determined by several straightforward properties, including the following support condition.

\begin{lemma}[Corollary 11.1.12 \cite{Kumar}]\label{lem:xisupport}
For any $\gamma,\eta \in P_{mn}$, we have $\xi^\gamma(\eta)=0$ if and only if $\gamma \not\subseteq \eta$.
\end{lemma}

\noindent The polynomials $\xi^\gamma(\eta)$ themselves can be computed in many different ways, including by the original formula of  Billey \cite{Billey}, formulated independently by Anderson, Jantzen, and Soergel \cite{AJS}. Ikeda and Naruse \cite{IkedaNaruse} interpret Billey's formula for the Grassmannian using excited Young diagrams; see Example \ref{ex:HLex} for an illustration of this method.   As we shall see, certain polynomials $\xi^\gamma(\eta)$ also coincide with special cases of equivariant Littlewood-Richardson coefficients.

\subsection{An equivariant Pieri rule via localization}\label{sec:PieriLoc}

Huang and Li proved the first equivariant quantum Pieri rule for any type $A$ partial flag variety in Theorem 3.10 of \cite{HuangLi}, a special case of which we review below as Theorem \ref{T:HuangLi}.  Their Pieri rule for the Grassmannian records the (classical) equivariant Littlewood-Richardson coefficient $c_{(1)^p,\mu}^{\lambda}$ by specifying a column-shape Schubert variety in a smaller Grassmannian, and localizing at a $T$-fixed point indexed by a partition obtained from $\lambda$ and $\mu$ via a join-and-cut algorithm.

More precisely, given a pair of partitions $\mu \subseteq \lambda \in P_{mn}$ such that the skew shape $\lambda/\mu=v^r$, Huang and Li define a new diagram $\lambda_\mu$ in the shorter rectangle $P_{m-r, n}$.  The closed formula for $\lambda_\mu$ is provided in \cite[Definition 3.14]{HuangLi}, and interpreted through a join-and-cut process explained in \cite[Definition 3.15]{HuangLi}. We now review the construction of $\lambda_\mu$, using a slight reformulation of the join-and-cut algorithm from \cite{HuangLi}. Given any $k \in [m]:=\{1,\dots, m\}$, the \emph{staircase} diagram with $k$ parts is defined as $\delta^k := (k-1,\dots, 1,0)$.

\begin{definition}\label{def:cutjoin}

Let $\mu=(\mu_1, \dots, \mu_m), \lambda \subseteq (\lambda_1, \dots, \lambda_m) \in P_{mn}$ be such that the skew shape $\lambda/\mu = v^r$ for some $0 \leq r \leq m$.   Construct a new partition by performing the following \emph{join-and-cut algorithm}:
\begin{enumerate}
\item Add the staircase $\delta^{m}$ to the diagram $\lambda$, where parts are added coordinate-wise.
\item Delete from the diagram resulting from step (1) those $r$ rows in which the original shapes $\lambda$ and $\mu$ differ; i.e. remove every row from $\lambda+\delta^m$ for which $\lambda_i \neq \mu_i$.
\item Remove the staircase $\delta^{m-r}$ from the shape resulting from step (2), where parts are subtracted coordinate-wise.
\end{enumerate}

The partition resulting from this join-and-cut algorithm is denoted by $\lambda_\mu \in P_{m-r,n}$. Equivalently, if we denote by $i_1 < \cdots < i_{m-r}$ those parts such that $\lambda_{i_\ell} = \mu_{i_\ell}$, we have
\begin{equation}\label{eq:cutjoin}
\lambda_\mu = (\mu_{i_1}-i_1+r+1, \mu_{i_2}-i_2+r+2, \dots, \mu_{i_{m-r}}-i_{m-r}+m) \in P_{m-r,n}.
\end{equation}
\end{definition}

We illustrate this join-and-cut algorithm in the following example. 

\begin{example}\label{ex:cutjoin}

Consider $\mu = (6,6,6,3,2,0,0) \subset \lambda = (7,6,6,4,2,1,0) \in P_{7,15}$, and recall that the skew shape $\lambda /\mu$ in the $7 \times 8$ rectangle is the vertical 3-strip consisting of the boxes containing a $\red{\sstar[1.5]}$ in the figure below.  We illustrate the join-and-cut algorithm, as described in Definition \ref{def:cutjoin}. Step (1) adds the staircase $\delta^7$ to $\lambda$ coordinate-wise. In the left-hand figure, the boxes of $\delta^7$ in the partition $\lambda + \delta^7$ are indicated with blue shading:
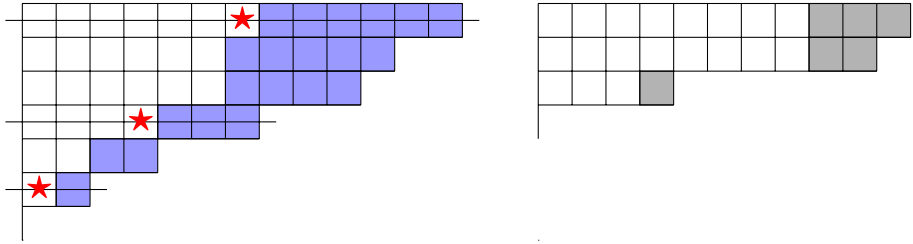
\[
\begin{tikzpicture}[scale=.45]

\path [draw, fill=blue!40] (21,8)--  (15,8) -- (15,7) -- (21,7);
\path [draw, fill=blue!40] (19,7)--  (14,7) -- (14,6) -- (19,6);
\path [draw, fill=blue!40] (18,6)--  (14,6) -- (14,5) -- (18,5);
\path [draw, fill=blue!40] (15,5)--  (12,5) -- (12,4) -- (15,4);
\path [draw, fill=blue!40] (12,4)--  (10,4) -- (10,3) -- (12,3);
\path [draw, fill=blue!40] (10,3)--  (9,3) -- (9,2) -- (10,2);

\draw[-]

(8,8) -- (21,8) 
(8,7) -- (21,7) 
(8,6) -- (19,6)
(8,5) -- (18,5) 
(8,4) -- (15,4) 
(8,3) -- (12,3)
(8,2) -- (10,2)
(8,1) -- (8,1)
(8,8) -- (8,1)
(9,8) -- (9,2)
(10,8) -- (10,2)
(11,8) -- (11,3)
(12,8) -- (12,3)
(13,8) -- (13,4)
(14,8) -- (14,4)
(15,8) -- (15,4)
(16,8) -- (16,5)
(17,8) -- (17,5)
(18,8) -- (18,5)
(19,8) -- (19,6)
(20,8) -- (20,7)
(21,8) -- (21,7)

(7.5,7.5) -- (21.5,7.5)
(7.5,4.5) -- (15.5,4.5)
(7.5,2.5) -- (10.5,2.5)
;

\node[color=red] at (14.5,7.5) {$\sstar[2]$};
\node[color=red] at (11.5,4.5) {$\sstar[2]$};
\node[color=red] at (8.5,2.5) {$\sstar[2]$};

\end{tikzpicture}
\quad\quad
\begin{tikzpicture}[scale=.45]

\path [draw, fill=darkgray!40] (19,8)--  (16,8) -- (16,7) -- (19,7);
\path [draw, fill=darkgray!40] (18,7)--  (16,7) -- (16,6) -- (18,6);
\path [draw, fill=darkgray!40] (12,6)--  (11,6) -- (11,5) -- (12,5);

\draw[-]

(8,8) -- (19,8) 
(8,7) -- (19,7) 
(8,6) -- (18,6)
(8,5) -- (12,5) 
(8,1) -- (8,1)
(8,8) -- (8,4)
(9,8) -- (9,5)
(10,8) -- (10,5)
(11,8) -- (11,5)
(12,8) -- (12,5)
(13,8) -- (13,6)
(14,8) -- (14,6)
(15,8) -- (15,6)
(16,8) -- (16,6)
(17,8) -- (17,6)
(18,8) -- (18,6)
(19,8) -- (19,7)
;

(7.5,7.5) -- (21.5,7.5)
(7.5,4.5) -- (15.5,4.5)
(7.5,2.5) -- (10.5,2.5)
;
\end{tikzpicture}\]
\[\lambda + \delta^7 = (13,11,10,7,4,2,0) \qquad \mapsto \qquad (11,10,4,0) \in P_{4,15}\]
Step (2) calls for removing the rows where $\lambda_i \not= \mu_i$; these 3 rows (each of which contains a $\red{\sstar[1.5]}$) are indicated in the left-hand figure using strike-through.
Finally, Step (3) requires us to remove the staircase $\delta^4$, indicated in the right-hand figure with gray shading, from which we obtain
\[ \lambda_\mu = (11,10,4,0) - \delta^4 = (8,8,3,0).\]
 For comparison, those parts such that $\lambda_{i_\ell} = \mu_{i_\ell}$ are given by  $\{i_1,i_2,i_3,i_4\} = \{2,3,5,7\}$, so that 
\[\lambda_\mu = (6-2+4, 6-3+5, 2-5+6, 0-7+7) = (8,8,3,0) \in P_{4,15}\]
by the closed formula \eqref{eq:cutjoin}. 
\end{example}

Together with any available formula for computing the localization $\xi^\gamma(\eta)$, the following (special case of a) theorem of Huang and Li provides a rule for computing products by Schubert classes indexed by column shapes in the equivariant cohomology of the Grassmannian.

\begin{theorem}[Equivariant Pieri Rule, Theorem 3.17 \cite{HuangLi}]\label{T:HuangLi}
For any integer $1 \leq p \leq m$ and any partition $\mu \in P_{m,n}$, we have
\begin{equation*}
\sigma_{(1)^p} \circ \sigma_\mu =\sum\limits_{0 \leq r \leq p}\, \sum\limits_ {\lambda/\mu = v^r} \xi^{(1)^{p-r}}_{m-r}(\lambda_\mu)\sigma_\lambda
\end{equation*}
in $H^*_{T}(Gr(m,n+1))$, where $\lambda \in P_{m,n+1}$.
\end{theorem}

\noindent The switch to $Gr(m,n+1)$ in Theorem \ref{T:HuangLi}, which persists throughout the remainder of the paper, is intentionally made to prepare for the application of the equivariant rim hook rule in the proof of Theorem \ref{thm:main}, in order to obtain an equivariant quantum Pieri rule for $Gr(m,n)$.  Note that Huang and Li also deliberately work in $Gr(m,n+1)$ throughout \cite{HuangLi}, albeit for different reasons.

For the sake of completeness, we now illustrate Theorem \ref{T:HuangLi} by calculating the same equivariant Littlewood-Richardson coefficient as in Example \ref{ex:thmex}. See \cite{TymBilley} for a survey on the localizations occurring in Theorem \ref{T:HuangLi}, including for more details on the method of excited Young diagrams illustrated in the following example.

\begin{example}\label{ex:HLex}
Let $p=5$, and recall that for $\mu = (6,6,6,3,2,0,0) \subset \lambda = (7,6,6,4,2,1,0) \in P_{7,15}$ with $r=3$, we have $\lambda_\mu = (8,8,3,0) \in P_{4,15}$ by Example \ref{ex:cutjoin}. 

Following \cite{IkedaNaruse}, we use the method of excited Young diagrams for calculating the localization $\xi^{(1)^2}_4(\lambda_\mu)$ required to determine the equivariant Littlewood-Richardson coefficient $c_{(1)^5,\mu}^{\lambda}$ using Theorem \ref{T:HuangLi}. In this case, there are 3 possible excited states corresponding to the column shape $(1)^2$ inside the diagram $\lambda_\mu \in P_{4,15}$, as follows:
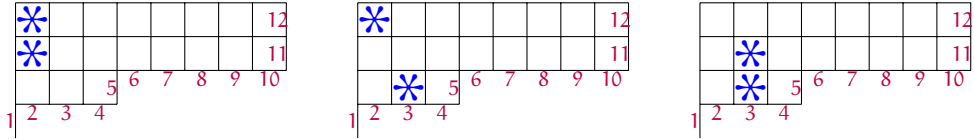
\[
\begin{tikzpicture}[scale=.45]

\draw[-]
(8,8) -- (16,8) 
(8,7) -- (16,7) 
(8,6) -- (16,6)
(8,5) -- (11,5) 
(8,8) -- (8,4)
(9,8) -- (9,5)
(10,8) -- (10,5)
(11,8) -- (11,5)
(12,8) -- (12,6)
(13,8) -- (13,6)
(14,8) -- (14,6)
(15,8) -- (15,6)
(16,8) -- (16,6);
\node[color=blue,font=\scriptsize] at (8.5,7.5) {$\sasterisk[2]$};
\node[color=blue,font=\scriptsize] at (8.5,6.5) {$\sasterisk[2]$};
\node[color=purple,font=\scriptsize] at (7.85,4.5) {$1$};
\node[color=purple,font=\scriptsize] at (8.5,4.75) {$2$};
\node[color=purple,font=\scriptsize] at (9.5,4.75) {$3$};
\node[color=purple,font=\scriptsize] at (10.5,4.75) {$4$};
\node[color=purple,font=\scriptsize] at (10.85,5.5) {$5$};
\node[color=purple,font=\scriptsize] at (11.5,5.75) {$6$};
\node[color=purple,font=\scriptsize] at (12.5,5.75) {$7$};
\node[color=purple,font=\scriptsize] at (13.5,5.75) {$8$};
\node[color=purple,font=\scriptsize] at (14.5,5.75) {$9$};
\node[color=purple,font=\scriptsize] at (15.5,5.75) {$10$};
\node[color=purple,font=\scriptsize] at (15.75,6.5) {$11$};
\node[color=purple,font=\scriptsize] at (15.75,7.5) {$12$};

\end{tikzpicture}
\hskip 15pt
\begin{tikzpicture}[scale=.45]

\draw[-]
(8,8) -- (16,8) 
(8,7) -- (16,7) 
(8,6) -- (16,6)
(8,5) -- (11,5) 
(8,8) -- (8,4)
(9,8) -- (9,5)
(10,8) -- (10,5)
(11,8) -- (11,5)
(12,8) -- (12,6)
(13,8) -- (13,6)
(14,8) -- (14,6)
(15,8) -- (15,6)
(16,8) -- (16,6);
\node[color=blue,font=\scriptsize] at (8.5,7.5) {$\sasterisk[2]$};
\node[color=blue,font=\scriptsize] at (9.5,5.5) {$\sasterisk[2]$};
\node[color=purple,font=\scriptsize] at (7.85,4.5) {$1$};
\node[color=purple,font=\scriptsize] at (8.5,4.75) {$2$};
\node[color=purple,font=\scriptsize] at (9.5,4.75) {$3$};
\node[color=purple,font=\scriptsize] at (10.5,4.75) {$4$};
\node[color=purple,font=\scriptsize] at (10.85,5.5) {$5$};
\node[color=purple,font=\scriptsize] at (11.5,5.75) {$6$};
\node[color=purple,font=\scriptsize] at (12.5,5.75) {$7$};
\node[color=purple,font=\scriptsize] at (13.5,5.75) {$8$};
\node[color=purple,font=\scriptsize] at (14.5,5.75) {$9$};
\node[color=purple,font=\scriptsize] at (15.5,5.75) {$10$};
\node[color=purple,font=\scriptsize] at (15.75,6.5) {$11$};
\node[color=purple,font=\scriptsize] at (15.75,7.5) {$12$};

\end{tikzpicture}
\hskip 15pt
\begin{tikzpicture}[scale=.45]

\draw[-]
(8,8) -- (16,8) 
(8,7) -- (16,7) 
(8,6) -- (16,6)
(8,5) -- (11,5) 
(8,8) -- (8,4)
(9,8) -- (9,5)
(10,8) -- (10,5)
(11,8) -- (11,5)
(12,8) -- (12,6)
(13,8) -- (13,6)
(14,8) -- (14,6)
(15,8) -- (15,6)
(16,8) -- (16,6);
\node[color=blue,font=\scriptsize] at (9.5,6.5) {$\sasterisk[2]$};
\node[color=blue,font=\scriptsize] at (9.5,5.5) {$\sasterisk[2]$};
\node[color=purple,font=\scriptsize] at (7.85,4.5) {$1$};
\node[color=purple,font=\scriptsize] at (8.5,4.75) {$2$};
\node[color=purple,font=\scriptsize] at (9.5,4.75) {$3$};
\node[color=purple,font=\scriptsize] at (10.5,4.75) {$4$};
\node[color=purple,font=\scriptsize] at (10.85,5.5) {$5$};
\node[color=purple,font=\scriptsize] at (11.5,5.75) {$6$};
\node[color=purple,font=\scriptsize] at (12.5,5.75) {$7$};
\node[color=purple,font=\scriptsize] at (13.5,5.75) {$8$};
\node[color=purple,font=\scriptsize] at (14.5,5.75) {$9$};
\node[color=purple,font=\scriptsize] at (15.5,5.75) {$10$};
\node[color=purple,font=\scriptsize] at (15.75,6.5) {$11$};
\node[color=purple,font=\scriptsize] at (15.75,7.5) {$12$};

\end{tikzpicture}
\]
Recording the up-steps and side-steps for $\lambda_\mu$ along the boundary of the diagram in red, the weight of an excited box (indicated in the figure above with a blue asterisk $\blue{\sasterisk[1.25]}$) is obtained directly from the corresponding up-step and side-step as $t_{u(\blue{\sasterisk[.75]})} - t_{s(\blue{\sasterisk[.75]})}$, using the conventions of this paper, which differ in sign from \cite{IkedaNaruse}.  For example, the weight of the top excited box in the left-hand figure is $t_{12}-t_2$. Altogether, we then have
\[c_{(1)^5,\mu}^{\lambda} = (t_{12}-t_2)(t_{11}-t_2)+(t_{12}-t_2)(t_{5}-t_3)+(t_{11}-t_3)(t_{5}-t_3)\]
in $H^*_{T}(Gr(7,15))$, where the weights are summed over the 3 excited states from left to right, with the weight of the top excited box recorded on the left within each product.
Most importantly for our purposes, we invite the reader to verify by regrouping terms that the above equivariant Littlewood-Richardson coefficient $c_{(1)^5,\mu}^{\lambda} = c_{(1)^5,\mu}^{\lambda,0}$ from Example \ref{ex:thmex}.
\end{example}

\begin{remark}\label{rmk:HLABnotbijective}
It is purely a coincidence that the number of excited states for the localization in Example \ref{ex:HLex} equals the number of configurations of addable boxes in Example \ref{ex:thmex}. In general, there are many more extensions $v^r \rightarrow v^p$ of the skew shape $\lambda/\mu$ than excited states of $(1)^{p-r}$ in the diagram $\lambda_\mu$, and equating individual equivariant Littlewood-Richardson coefficients via Theorems \ref{thm:main} and \ref{T:HuangLi} is an exercise in clever polynomial algebra.
\end{remark}

\subsection{Addable boxes as localizations}\label{sec:ABcutjoin}

In order to directly compare the equivariant Littlewood-Richardson coefficients in Theorem \ref{thm:main} to the corresponding localizations occurring in Theorem \ref{T:HuangLi}, we now explain how to calculate the weight of an addable box using the partition $\lambda_\mu$ resulting from the join-and-cut procedure, instead of the original skew shape $\lambda/\mu$. In light of Remark \ref{rmk:rectreduction}, throughout the remainder of this section, we typically refer to addable boxes within a fixed rectangle, as opposed to on the cylinder.

\begin{lemma}\label{lem:HLdiagram}

Fix an integer $1 \leq p \leq m$, and let $\mu=(\mu_1, \dots, \mu_m) \in P_{mn}$ be any partition.  Suppose that the skew shape $\lambda/\mu=v^r$ for some $\lambda=(\lambda_1, \dots, \lambda_m) \in P_{m,n+1}$ and $0\leq r \leq p$.  Denote by  $i_1< \cdots < i_{m-r}$ those parts such that $\lambda_{i_\ell} = \mu_{i_\ell}$, indexed from the top of the rectangle. 
Given any extension $v^r \rightarrow v^{p}$, consider an addable box $\alpha \in v^p\ba v^r$ in row $i_j$ of $\mu$ for some $j \in [m-r]$.  

Viewing $\lambda_\mu/\lambda_\mu$ as a vertical 0-strip $v^0$, then there is an addable box $\beta$ in row $j$ of $\lambda_\mu$, indexing rows from the top of the rectangle. As such, the given extension $v^r \rightarrow v^p$ corresponds uniquely to an extension $v^0 \rightarrow v^{p-r}$ of addable boxes in $\lambda_\mu$, identifying $\alpha \in v^p \ba v^r$ in row $i_j$ of $\mu$ with $\beta \in v^{p-r} \ba v^0$ in row $j$ of $\lambda_\mu$.
Moreover, under this correspondence,
\begin{equation}\label{eq:alphabetawt}
\wtv_{\mu}(\alpha) = \wtv_{\lambda_\mu}(\beta).
\end{equation}
\end{lemma}

\begin{proof}

We first make explicit the correspondence between extensions of vertical strips by addable boxes in $\mu$ and $\lambda_\mu$.
If for some $j \in [m-r]$, row $i_j$ of $\mu$ contains an addable box $\alpha$ in the extension $v^r \rightarrow v^p$, we first claim that row $j$ in $\lambda_\mu$ is also nonempty. Since the skew shape $\lambda/\mu$ supports an addable box in row $i_j$ counting from the top of the rectangle, then by definition $\lambda_{i_j} = \mu_{i_j} \neq 0$.   By \eqref{eq:cutjoin}, the number of boxes in row $j$ of $\lambda_\mu$, counting from the top of the rectangle, equals 
\begin{equation}\label{eq:lammuparts}
(\lambda_\mu)_j = \mu_{i_j} - i_j+r+j.
\end{equation}
The largest possible value for $i_j$ in the $m \times (n-m+1)$ rectangle occurs when the $m-r$ parts such that $\lambda_{i_j}  =\mu_{i_j}$ coincide with the  bottom $m-r$ rows of $\lambda/\mu$, which would then be indexed by
$i_1 = r+1 < \dots < i_{r-m} = m$. In this scenario, $i_j = r+j$, and so in general we necessarily have $i_j \leq r+j$.  Therefore, $-i_j +r+j \geq 0$, and since $\mu_{i_j} >0$, then \eqref{eq:lammuparts} implies that $(\lambda_\mu)_j >0$ as well.  In particular, whenever $\mu$ supports an addable box $\alpha$ in row $i_j$, then $\lambda_\mu$ supports a corresponding addable box $\beta$ in row $j$, indexing the rows from the top of the corresponding rectangles.  Identifying the addable box $\alpha \in \mu$ in row $i_j$ with the corresponding addable box $\beta \in \lambda_\mu$ in row $j$, the $p-r$ addable boxes in the extension $\lambda/\mu = v^r \rightarrow v^p$ correspond uniquely to an extension $v^0 \rightarrow v^{p-r}$ of addable boxes in $\lambda_\mu$.

We now proceed to calculate the weight of two corresponding addable boxes.  By definition, 
\begin{equation*}
\wtv_{\mu}(\alpha) =  t_{u(\alpha)}-t_{r(\alpha)-b(\alpha)} \quad \text{and} \quad \wtv_{\lambda_\mu}(\beta) =  t_{u(\beta)}-t_{r(\beta)-b(\beta)}.
\end{equation*}
First, recall that the row numbers $r(\alpha)$ and $r(\beta)$ are calculated from the bottom of their respective rectangles, which means that we can immediately calculate $r(\alpha) = m-i_j+1$ and $r(\beta) = m-r-j+1$.  

The index $u(\alpha)$ is the up-step for box $\alpha$, which can also be decomposed as the sum of the count of up-steps of $U(\mu)$ plus the count of side-steps of $S(\mu)$, up to and including the vertical edge of $\alpha$. Since $\alpha$ is in row $i_j$ of $\mu \in P_{mm}$ indexed from the top of the rectangle, the number of up-steps to this point is given by $m-i_j+1$. The number of side-steps to this point equals $\mu_{i_j}$, and so $u(\alpha) = \mu_{i_j}+m-i_j+1$.  Since $\beta$ is in row $j$ of $\lambda_\mu$ counted from the top of the rectangle, which only has $m-r$ rows, we instead have $u(\beta) = (\lambda_\mu)_j + m-r-j+1$.  Applying \eqref{eq:lammuparts}, we thus see that indeed
\[u(\beta) =\left(\mu_{i_j} - i_j+r+j\right) +(m-r-j+1)=u(\alpha).\]

We now compare the total number of addable and skew boxes below $\alpha$ and $\beta$, denoted $b(\alpha)$ and $b(\beta)$ respectively. First note that skew boxes do not occur in $\lambda_\mu$, since all rows in which $\lambda$ and $\mu$ have unequal parts are deleted in the join-and-cut procedure which produces $\lambda_\mu$. To carry out the calculation of $b(\beta)$, we define an intermediate statistic $n_k :=  \#\left\{i > k \mid \lambda_i \neq \mu_i \right\}$ which counts the number of rows below row $k$ in which $\lambda$ and $\mu$ have unequal parts, counting from the top of the rectangle.  Since exactly $r$ rows are deleted from $\mu$ in step (2) of Definition \ref{def:cutjoin}, and $n_{i_j}$ of these deleted rows lie below row $i_j$ in $\mu$, then $r-n_{i_j}$ rows are deleted from $\mu$ above row $i_j$.  Thus after the join-and-cut algorithm is carried out, the row number $r(\beta)$ for $\beta$ in $\lambda_\mu$ equals $m-i_j +1 - n_{i_j}$. Notice, though, that we already computed $r(\beta) = m-r-j+1$ above. Comparing these two calculations, we deduce that $n_{i_j}=-i_j+r+j$.

By construction, all of the boxes in the extension $v^r \rightarrow v^p$ that are below $\alpha$ remain below $\beta$ in the corresponding extension $v^0 \rightarrow v^{p-r}$.  Since the rows in which $\lambda_i \neq \mu_i$ are precisely those rows which are removed in the join-and-cut algorithm to produce $\lambda_\mu$, and those unequal rows below $i_j$ in the skew shape $\lambda/\mu$ each contribute one box to $b(\alpha)$, then by definition $b(\beta) = b(\alpha)-n_{i_j}$.  Using our formulas for $r(\alpha), r(\beta)$, and $n_{i_j}$ from above, we thus have
\begin{align*}
r(\beta) - b(\beta) & = \left( m-r-j+1 \right) - \left( b(\alpha)-n_{i_j} \right) \\
& = \left( m-r-j+1 \right) - b(\alpha)+ \left( -i_j+r+j \right) \\
& = \left( m-i_j+1 \right) - b(\alpha) \\
& = r(\alpha)-b(\alpha),
\end{align*}
as required to conclude our verification of \eqref{eq:alphabetawt}.
\end{proof}

We illustrate the correspondence defined by Lemma \ref{lem:HLdiagram} as follows.

\begin{example}\label{ex:wtnoskew}

Recall that for $\mu = (6,6,6,3,2,0,0) \subset \lambda = (7,6,6,4,2,1,0) \in P_{7,15}$, we have $\lambda/\mu = v^3$, and 
$\{i_1,i_2,i_3,i_4\} = \{2,3,5,7\}$ denote those parts in which $\lambda$ and $\mu$ agree.  Consider the extension $v^3 \rightarrow v^5$ given by the pair of addable boxes in rows $i_2=3$ and $i_3=5$, counting from the top of the rectangle, depicted on the left in the figure below.
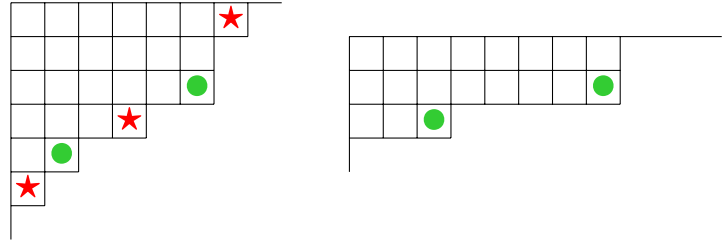
\begin{figure}[h]

\[\begin{tikzpicture}[scale=.45]

\draw[-]

(8,8) -- (16,8) 
(8,7) -- (15,7) 
(8,6) -- (14,6)
(8,5) -- (14,5) 
(8,4) -- (12,4) 
(8,3) -- (10,3) 
(8,2) -- (9,2) 

(8,8) -- (8,1)
(9,8) -- (9,2)
(10,8) -- (10,3)
(11,8) -- (11,4)
(12,8) -- (12,4)
(13,8) -- (13,5)
(14,8) -- (14,5)
(15,8) -- (15,7)
;

\node[color=red] at (14.5,7.5) {$\sstar[2]$};
\node[color=red] at (11.5,4.5) {$\sstar[2]$};
\node[color=red] at (8.5,2.5) {$\sstar[2]$};

\node at (9.5,3.5) {$\green{\sbullet[2]}$};
\node at (13.5,5.5) {$\green{\sbullet[2]}$};

\draw[-]

(18,7) -- (29,7) 
(18,6) -- (26,6)
(18,5) -- (26,5) 
(18,4) -- (21,4)

(18,7) -- (18,3)
(19,7) -- (19,4)
(20,7) -- (20,4)
(21,7) -- (21,4)
(22,7) -- (22,5)
(23,7) -- (23,5)
(24,7) -- (24,5)
(25,7) -- (25,5)
(26,7) -- (26,5)
;

\node at (20.5,4.5) {$\green{\sbullet[2]}$};
\node at (25.5,5.5) {$\green{\sbullet[2]}$};

\end{tikzpicture}\]

\end{figure}

By Lemma \ref{lem:HLdiagram}, the left-hand configuration of addable boxes corresponds to the right-hand configuration of addable boxes in the partition $\lambda_\mu = (8,8,3,0) \in P_{4,15}$ obtained by the join-and-cut algorithm in Example \ref{ex:cutjoin}.  In $\lambda_\mu$, the corresponding addable boxes are instead in rows 2 and 3, counting from the top of the shorter rectangle. Moreover, recall from Example \ref{ex:thmex} that 
\[\wtv_{\mu}(\alpha_{i_2}) = t_{11} - t_{(5-3)} = t_{11}-t_2\quad \text{and} \quad \wtv_{\mu}(\alpha_{i_3}) = t_{5} - t_{(3-1)} = t_5-t_2,\]
where here we denote by $\alpha_{i_j}$ the addable box in row $i_j$ of $\mu$, counting from the top of the rectangle. Similarly, we may compute that
\begin{align*}
\wtv_{\lambda_\mu}(\beta_2) &= t_{11} - t_{(3 - 1)} =t_{11} - t_2 = \wtv_{\mu}(\alpha_{i_2}), \\
\wtv_{\lambda_\mu}(\beta_3) & =  t_{5} - t_{(2 - 0)} =t_5 - t_2 = \wtv_{\mu}(\alpha_{i_3}),
\end{align*}
illustrating Equation \eqref{eq:alphabetawt}.
\end{example}

In light of Lemma \ref{lem:HLdiagram}, we aim next to repackage the weights on skew shapes in Theorem \ref{thm:main} in terms of weights on a partition shape, for which we require some auxiliary notation.
Recall that $1 \leq p \leq m$ is a fixed integer in the context of Theorem \ref{thm:main}.  For any fixed integer $0 \leq r \leq p$, throughout the rest of the paper, we write 
\[m':=m-r \quad \text{and} \quad p':=p-r.\]  
Define an indexing set
\begin{equation*}
\mathcal A_{p'} := \{ \vec{\iota} \in [m-p+1]^{p'} \mid \iota_1 \leq \cdots \leq \iota_{p'} \}.
\end{equation*}
The value $m-p+1$ arises here as the largest attainable row number (counting from the bottom of the rectangle) indexing the bottom-most addable box for a partition in $P_{m',n}$.

\begin{remark}\label{rem:iotadelta}

Given any vector $\vec{\iota}=(\iota_1, \dots, \iota_{p'}) \in \mathcal A_{p'}$, define 
\[ \vec{\iota}_\delta := \vec{\iota} +(0,1,\dots, p'-1),\] 
which adds the reverse of the staircase $\delta^{p'}$ to $\vec{\iota}$. Equivalently, writing $\vec{\iota}_\delta = (\iota^\delta_1, \dots, \iota^\delta_{p'})$, we see that
\[\iota_j^\delta:=\iota_j+j-1\]
for any $j \in [p'].$ Since $\iota_j \in [m-p+1]$, the entries of $\vec{\iota}_\delta$ are $p'$ distinct integers from the set $[m']$.  Conversely, any increasing choice $\vec{\iota}_\delta$ of $p'$ distinct integers in $[m']$ uniquely corresponds to $\vec{\iota} = \vec{\iota}_\delta - (0,1,\dots, p'-1) \in \mathcal A_{p'}$.  The correspondence $\vec{\iota} \longleftrightarrow \vec{\iota}_\delta$ is clearly a bijection between $\mathcal A_{p'}$ and the set $\mathcal A_{p'}^\delta$ of $p'$ strictly increasing integers in $[m']$.
\end{remark}

In the remainder of this section, we also fix a diagram $\eta \in P_{m',n+1}$. Recall that $U(\eta)$ denotes the set of up-steps of $\eta$, and enumerate these up-steps here as $\{u_1, \dots, u_{m'}\}$, recorded from the bottom of the rectangle.  Given any $\vec{\iota}=(\iota_1, \dots, \iota_{p'}) \in \mathcal A_{p'}$ and any index $j \in [p']$, define the weight
\begin{equation}\label{eq:Psideg1}
\Psi_{\eta}(\vec{\iota},j) := t_{u_{\iota_j^\delta}}-t_{\iota_j}.
\end{equation}
The following example illustrates how to calculate these weights.

\begin{example}\label{ex:Psimonomial}

Recall $\lambda_\mu = (8,8,3,0) \in P_{4,15}$ from Example \ref{ex:cutjoin}, which originated from a pair of partitions $\lambda, \mu \in P_{7,15}$ such that the skew shape $\lambda/\mu = v^3$.  Recalling that $m=7$ and fixing $p=5$ as in Example \ref{ex:thmex}, we have $m'=7-3=4$ and $p' = 5-3=2$, and so $\mathcal{A}_2 = \{ \vec{\iota} \in [3]^2 \mid \iota_1 \leq \iota_2 \}$.

As an example, fix $\vec{\iota}=(2,2) \in \mathcal A_2$.  We demonstrate how to calculate the weights $\Psi_{\lambda_\mu}(\vec{\iota},1)$ and $\Psi_{\lambda_\mu}(\vec{\iota},2)$. First, the vector $\vec{\iota}_\delta = \vec{\iota} + (0,1) = (2,3) \in \mathcal A_{2}^\delta$ indexes two rows in the partition $\lambda_\mu$, counting from the bottom of the rectangle.
The up-steps of $\lambda_\mu$ are
\[U(\lambda_\mu) = \{1,5,11,12\},\]
and so the up-steps isolated by $\vec{\iota}_\delta=(2,3)$ are $u_2 = 5$ and $u_3 = 11$.  Therefore, by \eqref{eq:Psideg1} we have 
\[ \Psi_{\lambda_\mu}(\vec{\iota},1) = t_5-t_2 \quad \text{and} \quad \Psi_{\lambda_\mu}(\vec{\iota},2) = t_{11}-t_2. \]

To take this calculation one step further, the vector $\vec{\iota}_\delta = (2,3)$ indexes two rows of $\lambda_\mu$, each of which supports an addable box as follows:

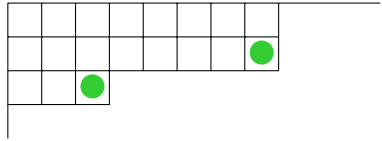
\begin{center}

\begin{tikzpicture}[scale=.45]

\draw[-]

(18,7) -- (29,7) 
(18,6) -- (26,6)
(18,5) -- (26,5) 
(18,4) -- (21,4)

(18,7) -- (18,3)
(19,7) -- (19,4)
(20,7) -- (20,4)
(21,7) -- (21,4)
(22,7) -- (22,5)
(23,7) -- (23,5)
(24,7) -- (24,5)
(25,7) -- (25,5)
(26,7) -- (26,5)
;

\node[font=\scriptsize] at (20.5,4.5) {$\green{\sbullet[3]}$};
\node[font=\scriptsize] at (25.5,5.5) {$\green{\sbullet[3]}$};

\end{tikzpicture}
\end{center}

\noindent Denote here by $\gamma_j$ the addable box in row $j$ of $\lambda_\mu$, counting from the bottom of the rectangle.  Then, comparing the addable boxes $\beta_j$ from Example \ref{ex:wtnoskew} which are indexed from the top down, we have
\[\wtv_{\lambda_\mu} (\gamma_2) = \wtv_{\lambda_\mu}(\beta_3) = t_5-t_2= \Psi_{\lambda_\mu}(\vec{\iota},1) \quad \text{and} \quad \wtv_{\lambda_\mu} (\gamma_3) = \wtv_{\lambda_\mu}(\beta_2) =t_{11}-t_2 =\Psi_{\lambda_\mu}(\vec{\iota},2).\]
We formalize this additional observation in the next lemma.
\end{example}

As illustrated by Example \ref{ex:Psimonomial}, the weights $\Psi_{\eta}(\vec{\iota},j)$ encode the weights of those addable boxes occurring in Lemma \ref{lem:HLdiagram} precisely as follows.

\begin{lemma}\label{lem:wtnoskew}
Let $\eta \in P_{m',n+1}$, and suppose that $\vec{\iota}_\delta = (\iota_1^\delta, \dots, \iota_{p'}^\delta) \in \mathcal A_{p'}^\delta$ indexes the rows of $\eta$ which contain addable boxes in an extension $\eta=v^0 \rightarrow v^{p'}$, counting rows from the bottom of the rectangle.   Then for $\vec{\iota} = \vec{\iota}_\delta - (0,1,\dots, p'-1) \in \mathcal A_{p'}$ and any $j \in [p']$,  
\begin{equation*}
\wtv_\eta(\beta_{\iota_j^\delta}) = \Psi_{\eta}(\vec{\iota},j),
\end{equation*}
where $\beta_{\iota_j^\delta}$ is the addable box in row $\iota_j^\delta$ of $\eta$, counting rows from the bottom of the rectangle. 
\end{lemma}

\begin{proof}
For ease of notation, throughout this proof we denote by $\beta := \beta_{\iota_j^\delta}$. Since $\beta$ is in row $\iota_j^\delta$ counting from the bottom of the rectangle, we immediately have $r(\beta) = \iota_j^\delta$. Having also enumerated $U(\eta) = \{u_1,\dots,u_{m^\prime}\}$ from the bottom of the rectangle, we automatically have $u(\beta) = u_{\iota^{\delta}_{j}}$.  
By hypothesis, the entries of $\vec{\iota}_\delta = (\iota_1^\delta, \dots, \iota_{p'}^\delta)$ index the rows in $\eta$ which contain addable boxes in an extension $v^0 \rightarrow v^{p'}$, counting from the bottom of the rectangle, and so $b(\beta) = j-1$.  Therefore, 
\[r(\beta)-b(\beta) = \iota_j^\delta - (j-1) = (\iota_j+j-1) -(j-1)= \iota_j,\]
and so  $\wtv_\eta(\beta)= t_{u(\beta)}-t_{r(\beta)-b(\beta)} = t_{u_{\iota_j^\delta}}-t_{\iota_j} =  \Psi_{\eta}(\vec{\iota},j),$
as claimed.
\end{proof}

Ultimately, the equivariant Pieri rule on cylindric shapes requires consideration of all possible extensions $v^r \rightarrow v^p$. We thus define the following sum
\begin{equation}\label{eq:Psiformula}
\Psi(\eta, p') := \sum\limits_{\vec{\iota} \in \mathcal{A}_{p'}}\prod\limits_{j \in [p']}\Psi_{\eta}(\vec{\iota},j),
\end{equation}
which yields the equivariant Littlewood-Richardson coefficients in Theorem \ref{thm:main} as follows.

\begin{proposition}\label{prop:PsiAB}

Fix an integer $1 \leq p \leq m$ and a partition $\mu \in P_{mn}$.  For any integer $0 \leq r \leq p$ and any partition $\lambda \in P_{m,n+1}$ such that the skew shape $\lambda/\mu = v^r$, we have
\begin{equation}\label{eq:PsiAB}
\Psi( \lambda_\mu,p') = \sum_{v^r \rightarrow v^p}\prod_{\alpha \in v^p \ba v^r}\wtv_{\mu}(\alpha).
\end{equation}
\end{proposition}

\begin{proof}
By Lemma \ref{lem:HLdiagram}, any extension $v^r \rightarrow v^p$ in $\mu$ corresponds to a unique extension $v^0 \rightarrow v^{p'}$ in $\lambda_\mu$, and the weights of corresponding addable boxes are equal.  Therefore, by Lemma \ref{lem:HLdiagram}, we have
\begin{equation}\label{eq:PsiAB1}
\sum_{v^r \rightarrow v^p}\prod_{\alpha \in v^p \ba v^r}\wtv_{\mu}(\alpha)  = \sum_{v^0 \rightarrow v^{p'}}\prod_{\beta \in v^{p'}\ba  v^0}\wtv_{\lambda_\mu}(\beta).
\end{equation}

We proceed to rephrase the right-hand expression in \eqref{eq:PsiAB1} in terms of the weights occurring in Lemma \ref{lem:wtnoskew}. 
By definition, any extension $v^0\rightarrow v^{p'}$ consists of $p'$ addable boxes in distinct, nonzero rows of $\lambda_\mu \in P_{m',n+1}$.  Recording these row numbers in increasing order, counted from the bottom of the rectangle, gives an element $\vec{\iota}_\delta \in \mathcal A_{p'}^\delta$ by Remark \ref{rem:iotadelta}.  Identifying the extension $v^0 \rightarrow v^{p'}$ with the vector $\vec{\iota}_\delta$, since row $\iota_j^\delta$ of $\lambda_\mu$ supports an addable box, note that $(\lambda_\mu)_{\iota_j^\delta} \neq 0$ for all $j \in [p']$.  

Note, that Lemma \ref{lem:wtnoskew} only applies to those rows of $\lambda_\mu$ which are nonzero, and so our first reformulation of \eqref{eq:PsiAB1} will reflect this constraint. Recalling that the number of boxes in each row weakly increases going upward from the bottom of the rectangle,  since the bottom-most row in the extension $v^0 \rightarrow v^{p'}$ is nonzero, so are all the rows above it.  Rephrasing this observation in terms of the parts of $\lambda_\mu$, we have $(\lambda_\mu)_{\iota_j^\delta} \neq 0$ for all $j \in [p']$ if and only if $(\lambda_\mu)_{\iota_1^\delta} \neq 0$.  In addition, since $\iota_j^\delta = \iota_j+j-1$, then for $j=1$ we simply have $\iota_1^\delta = \iota_1$. Lemma \ref{lem:wtnoskew} then permits us to convert the weights $\wtv_{\lambda_\mu}(\beta)$ into polynomials of the form $\Psi_{\lambda_\mu}(\vec{\iota},j)$, for those $\vec{\iota}$ such that $\vec{\iota}_\delta$ index nonzero rows of $\lambda_\mu$, as follows
\begin{equation}\label{eq:PsiAB2}
\sum_{v^0 \rightarrow v^{p'}}\prod_{\beta \in v^{p'} \ba v^0}\wtv_{\lambda_\mu}(\beta) = \sum\limits_{\substack{\vec{\iota} \in \mathcal{A}_{p'} \\ \left( \lambda_\mu \right)_{\iota_1} \neq 0}}\prod\limits_{j \in [p']}\Psi_{\lambda_\mu}(\vec{\iota},j).
\end{equation}

We now consider those $\vec{\iota} \in \mathcal A_{p'}$ such that $(\lambda_\mu)_{\iota_1}=0$, which are currrently omitted from the sum in the right-hand expression in \eqref{eq:PsiAB2}.  In particular, we claim that each such $\vec{\iota}$ contributes zero to this sum.  To see this, recall that $\Psi_{\lambda_\mu}(\vec{\iota},1) = t_{u_{\iota_1}}-t_{\iota_1}$. Since the up-steps $U(\lambda_\mu)$ are indexed from the bottom of the rectangle, and since the number of parts weakly increases going upward while $(\lambda_\mu)_{\iota_1}=0$, then the first up-steps $u_1, \dots, u_{\iota_1}$ all satisfy $u_k=k$.  In particular, $u_{\iota_1} = \iota_1$, and so $\Psi_{\lambda_\mu}(\vec{\iota},1) = t_{u_{\iota_1}}-t_{\iota_1} = t_{\iota_1}-t_{\iota_1} = 0$.  Since we take the product over all $j \in [p']$, we have
\begin{equation}\label{eq:Psizero}
\sum\limits_{\substack{\vec{\iota} \in \mathcal{A}_{p'} \\ \left( \lambda_\mu \right)_{\iota_1} = 0}}\prod\limits_{j \in [p']}\Psi_{\lambda_\mu}(\vec{\iota},j) =0.
\end{equation}

We can thus add these terms back into the sum in \eqref{eq:PsiAB2} without affecting the value:
\begin{align}
 \sum\limits_{\substack{\vec{\iota} \in \mathcal{A}_{p'} \\ \left( \lambda_\mu \right)_{\iota_1} \neq 0}}\prod\limits_{j \in [p']}\Psi_{\lambda_\mu}(\vec{\iota},j) &=  \sum\limits_{\substack{\vec{\iota} \in \mathcal{A}_{p'} \\ \left( \lambda_\mu \right)_{\iota_1} \neq 0}}\prod\limits_{j \in [p']}\Psi_{\lambda_\mu}(\vec{\iota},j)+ \sum\limits_{\substack{\vec{\iota} \in \mathcal{A}_{p'} \\ \left( \lambda_\mu \right)_{\iota_1} = 0}}\prod\limits_{j \in [p']}\Psi_{\lambda_\mu}(\vec{\iota},j) \nonumber \\
 & = \sum\limits_{\vec{\iota} \in \mathcal{A}_{p'}}\prod\limits_{j \in [p']}\Psi_{\lambda_\mu}(\vec{\iota},j) \nonumber \\
 & = \Psi(\lambda_\mu, p'). \label{eq:PsiAB3}
\end{align}
Equation \eqref{eq:PsiAB} now follows by combining \eqref{eq:PsiAB1},  \eqref{eq:PsiAB2}, and  \eqref{eq:PsiAB3}.
\end{proof}

We demonstrate Proposition \ref{prop:PsiAB} in the following example.

\begin{example}\label{ex:Psiproduct}
Fix $p=5$, and $\mu = (6,6,6,3,2,0,0) \in P_{7,15}$.  Then $\lambda = (7,6,6,4,2,1,0) \in P_{7,15}$ satisfies $\lambda/ \mu = v^3$.  Recall from Example \ref{ex:thmex} that 
\[  \sum_{v^r \rightarrow v^p}\prod_{\alpha \in v^p \ba v^r}\wtv_{\mu}(\alpha) =(t_{12}-t_3)(t_{11}-t_3) + (t_{12}-t_3)(t_{5}-t_2) + (t_{11}-t_2)(t_{5}-t_2), \]
by calculating the weights of the three possible extensions $v^3 \rightarrow v^5$ depicted in Figure \ref{fig:vstrip}.  

On the other hand, we may apply Proposition \ref{prop:PsiAB} to calculate the same polynomial.  Recall from Example \ref{ex:Psimonomial} that $p'=5-3=2$ and
\begin{equation*}
 \mathcal{A}_2 = \{ \vec{\iota} \in [3]^2 \mid \iota_1 \leq \iota_2 \} =  \{(1,1),(1,2),(1,3),(2,2),(2,3),(3,3)\}.
\end{equation*}
By Example \ref{ex:cutjoin}, the partition $\lambda_\mu = (8,8,3,0) \in P_{4,15}$, which has up-steps $U(\lambda_\mu) = \{1,5,11,12\}$. We shall compute
\[\Psi(\lambda_\mu,2)= \sum\limits_{\vec{\iota} \in \mathcal{A}_{2}}\Psi_{\lambda_\mu}(\vec{\iota},1)\Psi_{\lambda_\mu}(\vec{\iota},2).\]
First, consider any of the vectors of the form $\vec{\iota} = (1,k) \in \mathcal A_2$.  In each case, the weight $\Psi_{\lambda_\mu}(\vec{\iota},1)=t_1-t_1=0$, and so the three vectors $(1,1),(1,2),(1,3) \in \mathcal A_2$ contribute zero to the sum $\Psi({\lambda_\mu},2)$, illustrating Equation \eqref{eq:Psizero} from the proof of Proposition \ref{prop:PsiAB}.

Each of the remaining three vectors in $\mathcal A_2$ contributes a nonzero term. Recall from Example \ref{ex:Psimonomial} that we calculated the weights $\Psi_{\lambda_\mu}(\vec{\iota},j)$ for $\vec{\iota}=(2,2)$ and $j \in [2]$.  In particular, we see that
\[\Psi_{\lambda_\mu}((2,2),1)\Psi_{\lambda_\mu}((2,2),2) = (t_5 - t_2)(t_{11} - t_2).\] 
Similarly, applying Equation \eqref{eq:Psideg1} directly, we have
\begin{align*}
\Psi_{\lambda_\mu}((2,3),1)\Psi_{\lambda_\mu}((2,3),2) & = (t_5 - t_2)(t_{12} - t_3), \\
\Psi_{\lambda_\mu}((3,3),1)\Psi_{\lambda_\mu}((3,3),2) & = (t_{11} - t_3)(t_{12} - t_3).
\end{align*}
For reference, the configurations of addable boxes corresponding to the vectors $(2,2),(2,3),(3,3) \in \mathcal A_2$ via Lemma \ref{lem:wtnoskew} are recorded left to right, respectively, in the following figure:

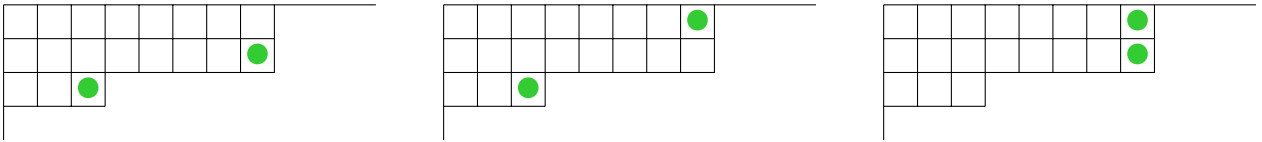
\begin{center}
\begin{tikzpicture}[scale=.45]

\draw[-]

(8,7) -- (19,7) 
(8,6) -- (16,6)
(8,5) -- (16,5) 
(8,4) -- (11,4) 

(8,7) -- (8,3)
(9,7) -- (9,4)
(10,7) -- (10,4)
(11,7) -- (11,4)
(12,7) -- (12,5)
(13,7) -- (13,5)
(14,7) -- (14,5)
(15,7) -- (15,5)
(16,7) -- (16,5)

(21,7) -- (32,7) 
(21,6) -- (29,6)
(21,5) -- (29,5) 
(21,4) -- (24,4) 

(21,7) -- (21,3)
(22,7) -- (22,4)
(23,7) -- (23,4)
(24,7) -- (24,4)
(25,7) -- (25,5)
(26,7) -- (26,5)
(27,7) -- (27,5)
(28,7) -- (28,5)
(29,7) -- (29,5)

(34,7) -- (45,7) 
(34,6) -- (42,6)
(34,5) -- (42,5) 
(34,4) -- (37,4) 

(34,7) -- (34,3)
(35,7) -- (35,4)
(36,7) -- (36,4)
(37,7) -- (37,4)
(38,7) -- (38,5)
(39,7) -- (39,5)
(40,7) -- (40,5)
(41,7) -- (41,5)
(42,7) -- (42,5)
;

\node at (10.5,4.5) {$\green{\sbullet[2]}$};
\node at (15.5,5.5) {$\green{\sbullet[2]}$};

\node at (23.5,4.5) {$\green{\sbullet[2]}$};
\node at (28.5,6.5) {$\green{\sbullet[2]}$};

\node at (41.5,5.5) {$\green{\sbullet[2]}$};
\node at (41.5,6.5) {$\green{\sbullet[2]}$};

\end{tikzpicture}
\end{center}

Taking the sum, we obtain
\begin{equation}\label{eq:Psifullex}
\Psi(\lambda_\mu,2) = (t_5 - t_2)(t_{11} - t_2)+  (t_5 - t_2)(t_{12} - t_3)+ (t_{11} - t_3)(t_{12} - t_3),
\end{equation}
which is  the same polynomial reviewed above from Example \ref{ex:thmex}, after reversing the order of the summands and multiplicands.
\end{example}

As we shall see in the next section, the polynomial $\Psi(\lambda_\mu,p')$ is the one which most naturally relates to the localizations appearing in Theorem \ref{T:HuangLi} of Huang and Li.  In particular, to prove Theorem \ref{thm:main}, we shall equate the polynomials occurring in Proposition \ref{prop:PsiAB} with Huang-Li's localization formula for equivariant Littlewood-Richardson coefficients from Theorem \ref{T:HuangLi} via the following key proposition.

\begin{proposition}\label{prop:PsiHL}
Fix an integer $1 \leq p \leq m$ and a partition $\mu \in P_{mn}$.  For any integer $0 \leq r \leq p$ and any partition $\lambda \in P_{m,n+1}$ such that the skew shape $\lambda/\mu = v^r$, we have
\begin{equation*}
 \xi^{(1)^{p'}}_{m'}(\lambda_\mu) =  \Psi(\lambda_{\mu},p').
\end{equation*}
\end{proposition}

Combining Proposition \ref{prop:PsiAB} with Proposition \ref{prop:PsiHL} proves the equality of all classical Littlewood-Richardson coefficients in Theorems \ref{thm:main} and \ref{T:HuangLi}. Since the quantum terms are easily obtained from the classical ones by applying the equivariant rim hook rule from \cite{BMT}, the proof of Proposition \ref{prop:PsiHL} will be the focus of the rest of the paper.

%%%%%%%%%%%%%%%%%%%%%%%%%%%%%%%%%%%%%%%%%%%%%%%%
%%%%%%%%%%%%%%%%%%%%%%%%%%%%%%%%%%%%%%%%%%%%%%%%
%%%%%%%%%%%%%%%%%%%%%%%%%%%%%%%%%%%%%%%%%%%%%%%%
%%%%%%%%%%%%%%%%%%%%%%%%%%%%%%%%%%%%%%%%%%%%%%%%

\section{Specializations of Factorial Schur Polynomials}\label{sec:FacSchurProof}

The proof of Proposition \ref{prop:PsiHL} is the primary goal of this final section of the paper. This proof proceeds by realizing in Section \ref{sec:spec} the localizations $\xi^\gamma(\eta)$ as specializations of factorial Schur polynomials, whose definition we review in Section \ref{sec:facSchurs}. The proof of Theorem \ref{thm:main} then immediately follows  in Section \ref{sec:proof}.

\subsection{Factorial Schur polynomials}\label{sec:facSchurs}

A \emph{semi-standard Young tableaux (SSYT) of shape $\lambda \in P_{mn}$} is a filling of each box of the Young diagram for $\lambda$ with a single number from the set $[m]$ such that numbers are weakly increasing across rows (reading from left to right) and strictly increasing down columns (reading from top to bottom).  Denote by $SSYT_\lambda$ the set of all SSYT of shape $\lambda$.  The \emph{Schur polynomial for $\lambda$} in the variables $(x)=x_1, \dots, x_m$ is then defined to be
\begin{equation}\label{eq:Schur}
s_\lambda(x) = \sum_{T_\lambda \in SSYT_\lambda} \prod_{\beta \in T_\lambda} x_{T_\lambda(\beta)},
\end{equation}
where $T_\lambda(\beta)$ denotes the number filling box $\beta$ in tableaux $T_\lambda$.
The elementary symmetric polynomial $e_i(x)$ is the Schur polynomial $s_{(1)^i}(x)$, and the homogeneous symmetric polynomial $h_i(x)$ equals $s_{(i)}(x)$.  The classical cohomology $H^*(Gr(m,n))$ has a well-known presentation as a polynomial ring generated by the elementary symmetric polynomials $e_1(x), \dots, e_m(x)$, in which the Schubert class $\sigma_\lambda$ is identified with the Schur polynomial $s_\lambda(x)$.

There is a corresponding presentation for $H^*_{T}(Gr(m,n))$ in terms of the equivariant analogs of the elementary and homogeneous symmetric polynomials; namely the \emph{factorial elementary symmetric polynomials} $e_i(x|t)$ and the \emph{factorial homogeneous symmetric polynomials} $h_i(x|t)$ in the variables $(x)$ and $(t)=t_1, \dots, t_n$.
 If we denote by $\Lambda := \Z[t_1, \dots, t_n]$, then the $T$-equivariant cohomology ring $H^*_{T}(Gr(m,n))$ has the following presentation 
\begin{equation*}
H^*_{T}(Gr(m,n)) \cong \frac{\Lambda[e_1(x|t), \dots, e_m(x|t)]}{\langle h_{n-m+1}(x|t), \dots, h_n(x|t)\rangle};
\end{equation*}
see \cite[Corollary 5.1]{MihalceaTrans}.
Moreover, under this isomorphism, the equivariant Schubert class $\sigma_\lambda$ is represented by the equivariant generalization of the Schur polynomial $s_\lambda(x)$, given by the associated factorial Schur polynomial; see  \cite[Proposition 5.2]{MihalceaTrans}.
For our purposes, we define the \emph{factorial Schur polynomial $s_\lambda(x|t)$} analogously to \eqref{eq:Schur} by
\begin{equation}\label{eq:facSchur}
s_\lambda(x|t) = \sum_{T_\lambda \in SSYT_\lambda} \prod_{\beta \in T_\lambda} x_{T_\lambda(\beta)}-t_{T_\lambda(\beta)+j-i},
\end{equation}
where the box $\beta \in T_\lambda$ is located in row $i$ and column $j$ of the tableaux, indexing rows from top to bottom and columns from left to right.  As in the classical case, we have $e_i(x|t) = s_{(1)^i}(x|t)$ and $h_i(x|t) = s_{(i)}(x|t)$.

\subsection{Specializations of factorial Schur polynomials}\label{sec:spec}

The localization $\xi^\gamma(\eta)$ of the equivariant Schubert class $\sigma_\gamma$ at the $T$-fixed point $\eta$ coincides with a certain specialization of the factorial Schur polynomial $s_\gamma(x|t)$.  We follow the treatment in \cite{IkedaNaruse}, though they credit Theorem \ref{thm:facSchur} below to earlier work of Knutson and Tao \cite{KT} and Lakshmibai, Raghavan, and Sankaran \cite{LRS}.

Any partition $\eta \in P_{mn}$ can be associated to a permutation in $S_n$, as we now review.  The \emph{window} notation for a permutation $\pi \in S_n$ records the action of $\pi$ on the set $[n]$ as $\pi = [\pi_1\, \cdots\, \pi_n]$, where $\pi_i:= \pi(i)$. A permutation $\pi \in S_n$ has a \emph{descent in position $k$} if $\pi(k) > \pi(k+1)$, and a permutation with at most one descent is called a \emph{Grassmann permutation}. The set of all Grassmann permutations with a descent in position $m$, together with the identity permutation, is denoted by $S^m_n$. 

 Given $\eta \in P_{mn}$, record the up-steps and side-steps of the partition from the bottom of the rectangle as $U(\eta)=\{u_1, \dots, u_m\}$ and $S(\eta) = \{s_1, \dots, s_{n-m}\}$. Define the associated permutation $\pi_\eta \in S_n^m$  by the following window
\begin{equation}\label{eq:perm}
\pi_\eta := [ u_1 \cdots u_m \,|\, s_1 \cdots s_{n-m}].
\end{equation}
Note that if $\eta$ is nonzero, then $\pi_\eta$ has exactly one descent in position $m$, as suggested by the vertical line in the window above. 
We define a collection $\left(x_{\pi_\eta} \right)$ of $m$ variables associated to the permutation $\pi_\eta$ by
\begin{equation}\label{eq:spec}
\left(x_{\pi_\eta} \right) :=  \left( t_{\pi_\eta(i)}\right)
\end{equation}
for $i \in [m]$. In particular, note that $\left(x_{\pi_\eta}\right)_i = t_{u_i}$ by definition.

\begin{theorem}[Theorem 5.4 \cite{IkedaNaruse}]\label{thm:facSchur}
For any partitions $\gamma \subseteq \eta \in P_{mn}$, we have
\begin{equation*}
\xi^\gamma(\eta) = s_\gamma(x_{\pi_\eta} | t)
\end{equation*}
\end{theorem}

We remark that the above statement differs from \cite[Theorem 5.4]{IkedaNaruse} by the sign ($-1)^{|\gamma|}$ because our convention for torus weights follows \cite{BMT}, in order to apply the equivariant rim hook rule in the proof of Theorem \ref{thm:main} later in Section \ref{sec:proof}. Our next and final lemma provides the critical link required to prove Proposition \ref{prop:PsiHL}, and thus the main theorem.

\begin{lemma}\label{lem:Psielem}
Fix any integers $1 \leq p \leq m$ and $0 \leq r \leq p$. 
For any partition $\eta \in P_{m',n+1}$, we have
\begin{equation*}
\Psi(\eta, p') = e_{p'}(x_{\pi_\eta}|t),
\end{equation*}
where $e_{p'}(x|t)$ is the factorial elementary symmetric polynomial in $x_1, \dots, x_{m'}$ and $t_1, \dots, t_{n+1}$.
\end{lemma}

Note that Example \ref{ex:SSYT} below illustrates both the statement of Lemma \ref{lem:Psielem}, as well as the weight-preserving bijection which is the key ingredient in the proof.

\begin{proof} 
For notational convenience throughout this proof alone, we shall write $k:=p'$. 
We thus aim to show that $\Psi(\eta,k) = e_k(x_{\eta_\pi}|t)$ for any $\eta \in P_{m',n+1}$.

Recalling that $e_k(x|t)$ is the factorial Schur polynomial $s_{(1)^k}(x|t)$, formula \eqref{eq:facSchur} gives us
\begin{equation}\label{eq:eSSYT}
e_k(x|t) = \sum_{\vec{\ell} \in SSYT_{(1)^k}} \prod_{j=1}^k x_{\ell_j}-t_{\ell_j-j+1}.
\end{equation}
Here $SSYT_{(1)^k}$ denotes the set of all strictly increasing fillings of the column shape containing $k$ boxes using the alphabet $[m']$.  By recording only the fillings, we may thus identify $SSYT_{(1)^k}$ with the set $\mathcal A_k^\delta = \{ \vec{\ell}\in [m']^k \mid \ell_1 < \cdots < \ell_k \}$ from Remark \ref{rem:iotadelta}.  

For any $\vec{\ell} \in SSYT_{(1)^k}$, denote the corresponding summand by 
$e^{\vec{\ell}}_k(x|t) =  \prod_{j=1}^k x_{\ell_j}-t_{\ell_j-j+1}.$
 Enumerate the up-steps of $\eta$ as $U(\eta) = \{u_1, \dots, u_{m'}\}$, recorded from the bottom of the rectangle.
Specializing the variables $x_i \mapsto t_{\eta_\pi(i)} = t_{u_i}$ as prescribed by $x_{\eta_\pi}$ via \eqref{eq:spec}, we obtain
\begin{equation}\label{eq:ellSSYT}
e^{\vec{\ell}}_k(x_{\eta_\pi}|t) =  \prod_{j=1}^k t_{u_{\ell_j}}-t_{\ell_j-j+1}.
\end{equation}

On the other hand, since $SSYT_{(1)^k}$ and $\mathcal{A}_k^{\delta}$ are in bijection, the set $SSYT_{(1)^k}$ is also in bijection with $\mathcal A_k =  \{ \vec{\iota} \in [m-p+1]^{k} \mid \iota_1 \leq \cdots \leq \iota_{k} \}$ by Remark \ref{rem:iotadelta}. If we fix any $\vec{\iota}=(\iota_1, \dots, \iota_{k}) \in \mathcal A_{k}$, then the corresponding summand in expression \eqref{eq:Psiformula} for $\Psi(\eta, k)$ may be recorded via \eqref{eq:Psideg1} as 
\begin{equation*}
\Psi^{\vec{\iota}}(\eta, k) = \prod\limits_{j =1}^k t_{u_{\iota_j^\delta}}-t_{\iota_j},
\end{equation*}
where we recall from Remark \ref{rem:iotadelta} that $\iota_j^\delta=\iota_j+j-1$, and more generally that $\vec{\iota}_\delta = (\iota^\delta_1, \dots, \iota^\delta_{k}) = \vec{\iota}+(0,1,\dots, k-1)$. In particular, the map $\vec{\iota} \mapsto \vec{\iota}_\delta$ gives the bijection from $\mathcal A_k$ to $SSYT_{(1)^k}$ noted above. 

Now given any $\vec{\iota} \in \mathcal A_k$ and $j \in [k]$, consider the corresponding multiplicand $t_{u_{\iota_j^\delta}}-t_{\iota_j}$ in the product $\Psi^{\vec{\iota}}(\eta, k) $.  The $j^{\text{th}}$ multiplicand in the associated product $ e^{\vec{\iota}_\delta}_k(x_{\eta_\pi}|t)$ is then $ t_{u_{\iota^\delta_j}}-t_{\iota^\delta_j-j+1}=t_{u_{\iota_j^\delta}}-t_{\iota_j}$, since  $\iota_j^\delta=\iota_j+j-1$. Therefore, $\Psi^{\vec{\iota}}(\eta, k)  = e^{\vec{\iota}_\delta}_k(x_{\eta_\pi}|t)$ for any $\vec{\iota} \in \mathcal A_k$. Summing over all $\mathcal A_k$ and $SSYT_{(1)^k}$, respectively, we then have $\Psi(\eta,k)=e_k(x_{\eta_\pi}|t)$, which proves the lemma.  
\end{proof}

We now revisit our running example from Section \ref{sec:Localizations}, to illustrate both the statement of Lemma \ref{lem:Psielem} and the method of proof.

\begin{example}\label{ex:SSYT}
Fix $p=5$ and $r=3$, in which case $p' = 5-3=2$.  Consider $\eta = (8,8,3,0) \in P_{4,15}$, and recall from Example \ref{ex:cutjoin} that $\eta = \lambda_\mu$ for partitions\ $\mu \subseteq \lambda \in P_{7,15}$ so that here $m' = 7-3=4$.  Recall by \eqref{eq:Psifullex} in Example \ref{ex:Psiproduct} that
\begin{equation}\label{eq:Psiexample}
\Psi(\eta,2) = (t_5 - t_2)(t_{11} - t_2)+  (t_5 - t_2)(t_{12} - t_3)+ (t_{11} - t_3)(t_{12} - t_3).
\end{equation}

On the other hand, we now compute $e_2(x_{\pi_\eta}\vert t)$ directly from the definition.  The set $SSYT_{(1)^2}$ consists of the 6 strictly increasing fillings of the column shape $(1)^2$ using the alphabet $[m'] = [4]$, which we depict as
\[
\begin{tikzpicture}[scale=.45]

\draw[-]
(1,1) -- (1,3) -- (2,3) -- (2,1) -- cycle
(1,2) -- (2,2)

(3,1) -- (3,3) -- (4,3) -- (4,1) -- cycle
(3,2) -- (4,2)

(5,1) -- (5,3) -- (6,3) -- (6,1) -- cycle
(5,2) -- (6,2)

(7,1) -- (7,3) -- (8,3) -- (8,1) -- cycle
(7,2) -- (8,2)

(9,1) -- (9,3) -- (10,3) -- (10,1) -- cycle
(9,2) -- (10,2)

(11,1) -- (11,3) -- (12,3) -- (12,1) -- cycle
(11,2) -- (12,2)
;

\node[scale=1] at (1.5,2.5) {$1$};
\node[scale=1] at (1.5,1.5) {$2$};

\node[scale=1] at (3.5,2.5) {$1$};
\node[scale=1] at (3.5,1.5) {$3$};

\node[scale=1] at (5.5,2.5) {$1$};
\node[scale=1] at (5.5,1.5) {$4$};

\node[scale=1] at (7.5,2.5) {$2$};
\node[scale=1] at (7.5,1.5) {$3$};

\node[scale=1] at (9.5,2.5) {$2$};
\node[scale=1] at (9.5,1.5) {$4$};

\node[scale=1] at (11.5,2.5) {$3$};
\node[scale=1] at (11.5,1.5) {$4$};

\end{tikzpicture}\ \ .
\]
If we encode each of these SSYT as a vector $\vec{\ell} = (\ell_1,\ell_2)$ where $\ell_1 < \ell_2$ denotes the filling, then by \eqref{eq:eSSYT} we have
\begin{equation}\label{eq:eSSYTex}
e_2(x|t) = \sum_{\vec{\ell} \in SSYT_{(1)^2}} \prod_{j=1}^2 x_{\ell_j}-t_{\ell_j-j+1}.
\end{equation}

Recall that for $\eta = (8,8,3,0) \in P_{4,15}$, we have $U(\eta) = \{1,5,11,12\} = \{u_1,u_2,u_3,u_4\}$ so that 
\[\pi_\eta = [1,5,11,12 \mid 2, 3,4,6,7,8,9,10,13,14,15]\]
by formula \eqref{eq:perm}.
Specializing the variables $x_i \mapsto t_{\pi_\eta(i)}$ in this example thus sends $x_1 \mapsto t_1$.  Therefore, the 3 left-hand SSYT, all of which satisfy $\ell_1 = 1$, correspond to a product having $(t_1-t_1)$ as a factor; that is, these 3 left-hand SSYT contribute zero to the sum in expression \eqref{eq:eSSYTex} for $e_2(x_{\pi_\eta}\vert t)$.  Recall from Example \ref{ex:Psiproduct} that  the polynomial $\Psi(\eta,2)$ also contains 3 zero terms corresponding to the vectors $(1,1),(1,2),(1,3) \in \mathcal A_2$. Moreover, for these 3 vectors $\vec{\iota} \in \mathcal A_2$, the associated vector $\vec{\iota}_\delta \in \mathcal A_2^\delta$ gives the fillings of the 3 left-hand SSYT.

For the 3 right-hand SSYT, note more generally that $x_i \mapsto t_{u_i}$ when we specialize $e_2(x_{\pi_\eta}\vert t)$. Indexing the summands of $e_2(x_{\pi_\eta}|t)$ as in \eqref{eq:ellSSYT} then gives us
\begin{align*}
e^{(2,3)}_2(x_{\pi_\eta}|t) &= (t_{u_{2}} - t_{2-1+1})(t_{u_3} - t_{3-2+1}) = (t_5 - t_2)(t_{11} - t_2),\\
e^{(2,4)}_2(x_{\pi_\eta}|t) &=  (t_{u_{2}} - t_{2-1+1})(t_{u_4} - t_{4-2+1}) = (t_5 - t_2)(t_{12} - t_3), \\
e^{(3,4)}_2(x_{\pi_\eta}|t) & = (t_{u_{3}} - t_{3-1+1})(t_{u_4} - t_{4-2+1}) = (t_{11} - t_3)(t_{12} - t_3).
\end{align*}
Taking the sum of the three above expressions and comparing with \eqref{eq:Psiexample}, we see that indeed $e_2(x_{\pi_\eta}\vert t) = \Psi(\eta,2)$ term by term, confirming the statement of Lemma \ref{lem:Psielem} in this example. Moreover, the 3 nonzero terms in $\Psi(\eta,2)$ correspond to the 3 vectors $(2,2),(2,3),(3,3) \in \mathcal A_2$ from Example \ref{ex:Psiproduct}, whose respective associates in $\mathcal A_2^\delta$ via Remark \ref{rem:iotadelta} give the fillings $(2,3),(2,4),(3,4)$ of the 3 right-hand SSYT above, completing our illustration of the weight-preserving bijection in the proof of Lemma \ref{lem:Psielem}.
\end{example}

Amusingly, there is a dual version of Lemma \ref{lem:Psielem}, obtained by noting that $\mathcal A_k$ is also in natural bijection with the set $SSYT_{(k)}$ of row-shaped semi-standard Young tableaux. This observation gives us a similar statement for the complete homogeneous factorial Schur functions, although the localization is dual in a sense, and the geometric interpretation is less clear.

\begin{remark}\label{lem:Psihomog}
Fix any integers $1 \leq p \leq m$ and $0 \leq r \leq p$.  For any partition $\eta \in P_{m',n+1}$, we have
\begin{equation*}
\Psi(\eta, p') =  (-1)^{p'} h_{p'}(x_{id} | t_{\pi_\eta}),
\end{equation*}
where $h_{p'}(x|t)$ is the factorial homogeneous symmetric polynomial in $x_1, \dots, x_{m'}$ and $t_1, \dots, t_{n+1}$, and we instead specialize the torus weights via $t_i \mapsto t_{\pi_\eta(i)}$, whereas $x_i \mapsto t_i$.
\end{remark}

\subsection{Proof of the equivariant quantum Pieri rule on cylindric shapes}\label{sec:proof}

We are now able to prove the key proposition upon which the proof of Theorem \ref{thm:main} relies.

\begin{proof}[Proof of Proposition \ref{prop:PsiHL}]
Fix an integer $1 \leq p \leq m$ and a partition $\mu \in P_{mn}$.  Consider any partition $\lambda \in P_{m,n+1}$ such that the skew shape $\lambda/\mu = v^r$ for some integer $0 \leq r \leq p$.  Apply Lemma \ref{lem:Psielem} to $\lambda_\mu \in P_{m',n+1}$ to obtain 
\begin{equation}\label{eq:Psielem}
\Psi(\lambda_\mu, p') =  e_{p'}(x_{\pi_{(\lambda_\mu)}}|t).
\end{equation}
In the case $(1)^{p'}\subseteq \lambda_\mu$, applying Theorem \ref{thm:facSchur} to \eqref{eq:Psielem}, we thus have
\begin{equation*}
\Psi(\lambda_\mu, p') =  e_{p'}(x_{\pi_{(\lambda_\mu)}}|t)= s_{(1)^{p'}}(x_{\pi_{(\lambda_\mu)}} | t) = \xi^{(1)^{p'}}_{m'}(\lambda_\mu),
\end{equation*}
confirming Proposition \ref{prop:PsiHL} in this case.

If $(1)^{p'} \not\subseteq \lambda_\mu$, then by Lemma \ref{lem:xisupport}, we know that $\xi^{(1)^{p'}}_{m'}(\lambda_\mu)=0$.  The condition $(1)^{p'} \not\subseteq \lambda_\mu$ implies that the number of nonzero parts of $\lambda_\mu \in P_{m',n+1}$ is strictly less than $p'$. In particular, for all $\vec{\iota} \in \mathcal A_{p'}$, we have $(\lambda_\mu)_{\iota_1} = 0$.  By \eqref{eq:Psizero} and \eqref{eq:PsiAB3} in the proof of Proposition \ref{prop:PsiAB}, we know that $\Psi(\lambda_\mu, p')=0$ in this case. Therefore, if $(1)^{p'} \not\subseteq \lambda_\mu$, we have 
\begin{equation*}
0=\Psi(\lambda_\mu, p') =  \xi^{(1)^{p'}}_{m'}(\lambda_\mu), 
\end{equation*}
as required to complete the proof of Proposition \ref{prop:PsiHL}.
\end{proof}

Equipped with Proposition \ref{prop:PsiHL}, we are now prepared to prove the main theorem of this paper.

\begin{proof}[Proof of Theorem \ref{thm:main}]
We first prove the rule for multiplying by a Schubert class indexed by a column-shaped partition. Fix an integer $1 \leq p \leq m$ and a partition $\mu \in P_{mn}$.    Consider any partition $\lambda \in P_{m,n+1}$ such that the skew shape $\lambda/\mu = v^r$ for some integer $0 \leq r \leq p$. Applying Proposition \ref{prop:PsiAB}, Proposition \ref{prop:PsiHL}, and Theorem \ref{T:HuangLi} in turn, we obtain
\begin{equation*}
 \sum_{v^r \rightarrow v^p}\prod_{\alpha \in v^p \ba v^r}\wtv_{\mu}(\alpha) = \Psi( \lambda_\mu,p') =  \xi^{(1)^{p'}}_{m'}(\lambda_\mu) = c_{(1)^p,\mu}^{\lambda},
\end{equation*}
where $c_{(1)^p,\mu}^{\lambda}$ is an equivariant Littlewood-Richardson coefficient in $H^*_{T}(Gr(m,n+1))$.
Further, note that in the Pieri case, $\sigma_{(1)^p} \circ \sigma_\mu \in H^*_{T}(Gr(m,n+1))$ is the same as the product $\sigma_{(1)^p} \circ \sigma_\mu \in H^*_{T}(Gr(m,2n-1))$ required to apply the equivariant rim hook rule from \cite{BMT}, reviewed here as Theorem \ref{thm:eqrimhook}.

The argument now naturally divides into two cases: $\lambda \in P_{mn}$ and $\lambda \notin P_{mn}$.
If $\lambda \in P_{mn}$, then by the equivariant rim hook rule, $c_{(1)^p,\mu}^{\lambda} = c_{(1)^p,\mu}^{\lambda,0}$ for $\sigma_{(1)^p} \star \sigma_\mu \in QH^*_{T}(Gr(m,n))$ as well, since rim hook reduction acts as the identity both on $\sigma_\lambda$ and the torus weights $t_1, \dots, t_n$ in this case. The formula for the classical terms in Theorem \ref{thm:main} thus follows.

Now suppose that $\lambda \notin P_{mn}$.   Since the parts of $\lambda=(\lambda_1, \dots, \lambda_m)$ are weakly decreasing and $\lambda \in P_{m,n+1}\backslash P_{mn}$, we must have $\lambda_1 = n-m+1$.  Recall by Theorem \ref{thm:eqrimhook} that if the $n$-core $\nu$ of $\lambda$ is not in $P_{mn}$, then the equivariant rim hook rule sends $\sigma_\lambda \mapsto 0$, and this term does not contribute to the product $\sigma_{(1)^p} \star \sigma_\mu \in QH^*_{T}(Gr(m,n))$. We thus reduce to the case where $\nu \in P_{mn}$. The $n$-core $\nu \in P_{mn}$ in this case if and only if the partition $\lambda \in P_{m,n+1}$ contains a removable $n$-rim hook, in which case we must have $\lambda_m \geq 1$ as well.  In the Pieri case, we can remove exactly one $n$-rim hook of height $m$, so that in the statement of Theorem \ref{thm:eqrimhook}, we have $\varepsilon_1 = m$.  The equivariant rim hook rule thus sends $\sigma_\lambda \mapsto (-1)^{\varepsilon_1-m} q\sigma_\nu = q \sigma_\nu$.  Further, recall from our discussion in Section \ref{sec:QPieri} that since $\lambda$ and $\nu$ are related by removing a single $n$-rim hook with these coordinates, then the closed loops $\lambda[0]=\nu[1]$ on the cylinder $\mathcal C_{m,n+1}$.

Since $\lambda \in P_{m,n+1}$, for any addable box $\alpha$ in an extension $v^r \rightarrow v^p$, the row number satisfies $r(\alpha) \leq m<n$, and the number of boxes below $\alpha$ also satisfies $b(\alpha) < m<n$.   In addition, since $\mu \in P_{mn}$, while $\lambda_1 = n-m+1$, we know that the skew shape $\lambda/\mu$ contains a box in the top row of the rectangle.  In particular, any addable box $\alpha \in v^p \ba v^r$ sits in or below row $m-1$, indexing rows from the bottom of the rectangle.  Therefore, all up-step indices satisfy $u(\alpha) \leq n-1<n$.
Since the equivariant rim hook rule sends $t_i \mapsto t_{i(\text{mod}\, n)}$ and each of these three statistics is strictly less than $n$ for any addable box $\alpha \in v^p \ba v^r$, each torus weight $t_{u(\alpha)}-t_{r(\alpha)-b(\alpha)}$ occurring in $c_{(1)^p,\mu}^\lambda$ is mapped to itself by the equivariant rim hook rule. Therefore, in case $\lambda \notin P_{mn}$, we have
\begin{equation*}
c_{(1)^p,\mu}^{\lambda} =  \sum_{v^r \rightarrow v^p } \prod_{\alpha \in v^p \ba v^r}\wtv_{\mu}(\alpha) =c_{(1)^p,\mu}^{\nu,1},
\end{equation*}
where the skew shape $v^r = \lambda/\mu = \nu[1]/\mu[0] = \nu/1/\mu$.  The formula for the quantum terms in the product $\sigma_{(1)^p} \star \sigma_\mu \in QH^*_{T}(Gr(m,n))$ in Theorem \ref{thm:main} thus follows by re-indexing.

We now deduce the rule for multiplying by a Schubert class indexed by a row-shaped partition from the column-shape rule via level-rank duality. Fix an integer $1 \leq k \leq n-m$ and a partition $\mu \in P_{mn}$.    Consider any partition $\lambda \in P_{mn}$ such that the skew shape $\lambda/d/\mu = h^r$ for some integer $0 \leq r \leq k$ and $d \in \{0,1\}$.  Let $\alpha \in h^k \ba h^r$ be an addable box in an extension $h^r \rightarrow h^k$. 

Since $\lambda/d/\mu$ is a horizontal $r$-strip, then the transposed skew shape $\lambda'/d/\mu'$ is a vertical $r$-strip $v^r$ with a corresponding extension $v^r \rightarrow v^k$.  Moreover, the addable box $\alpha \in h^k \ba h^r$ corresponds to a unique addable box $\beta \in v^k \ba v^r$ in the image of the transpose map.  Since $\alpha$ and $\beta$ are related by the transpose, by Definition \ref{def:addboxstats} we automatically have $r(\beta) = c(\alpha)$ and $b(\beta) = rt(\alpha)$, and therefore $r(\beta)-b(\beta) = c(\alpha)-rt(\alpha)$. In addition, the up-step $u(\beta)$ corresponds precisely to the side-step $s(\alpha)$ under the transpose map, both of which are enumerated from the bottom of the rectangle, and thus $u(\beta)+s(\alpha)=n+1$.

Under the involution on torus weights from level-rank duality as in Theorem \ref{thm:levelrank} we now have
\[-w_0: 
\begin{cases}
\quad \ \ \  \quad t_{u(\beta)} \mapsto -t_{s(\alpha)} \\
-t_{r(\beta)-b(\beta)} \mapsto t_{n+1-(c(\alpha)-rt(\alpha))}.
\end{cases}
\]
Comparing with Definition \ref{def:wtaddbox} gives us $-w_0: \wtv_{\mu'}(\beta) \mapsto \wth_{\mu}(\alpha)$.  
Having already proved above that
\[ c_{(1)^k,\mu'}^{\lambda',d} = \sum\limits_{v^r\rightarrow v^k} \prod_{\beta \in v^k \ba v^r} \wtv_{\mu'}(\beta)\]
in $QH^*_T(Gr(n-m,n))$, applying the level-rank duality from Theorem \ref{thm:levelrank} to both sides of this expression yields
\[ c_{(k),\mu}^{\lambda,d} = \sum\limits_{h^r\rightarrow h^k} \prod_{\alpha \in h^k \ba h^r} \wth_{\mu}(\alpha),\]
as required to complete the proof of the formula for $\sigma_{(k)} \star \sigma_{\mu} \in QH^*_T(Gr(m,n))$ in Theorem \ref{thm:main}.
\end{proof}

\bibliographystyle{alphanum}

\bibliography{BEMT-CylindricEqQPieri-References}

\end{document}